\documentclass[12pt, reqno]{amsart}
\usepackage{amssymb,amsthm,amsmath}
\usepackage[numbers,sort&compress]{natbib}
\usepackage{amssymb,amsmath}
\usepackage{amsfonts}
\usepackage{mathrsfs}
\usepackage{latexsym}
\usepackage{amssymb}
\usepackage{amsthm}
\usepackage{indentfirst}
\date{June 8, 2014}

\let\oldsection\section
\renewcommand\section{\setcounter{equation}{0}\oldsection}

\newtheorem{theorem}{Theorem}[section]
\newtheorem{lemma}{Lemma}[section]
\newtheorem{proposition}{Proposition}[section]
\newtheorem{definition}{Definition}[section]

\newtheorem{remark}{Remark}[section]

\begin{document}

\title[3D Primitive Equations with Horizontal Viscosities and Diffusion]{Global Well-posedness of the 3D Primitive Equations with Only  Horizontal Viscosity and Diffusion}

\author{Chongsheng~Cao}
\address[Chongsheng~Cao]{Department of Mathematics, Florida International University, University Park, Miami, FL 33199, USA}
\email{caoc@fiu.edu}

\author{Jinkai~Li}
\address[Jinkai~Li]{Department of Computer Science and Applied Mathematics, Weizmann Institute of Science, Rehovot 76100, Israel}
\email{jklimath@gmail.com}

\author{Edriss~S.~Titi}
\address[Edriss~S.~Titi]{Department
of Computer Science and Applied Mathematics, Weizmann Institute of Science,
Rehovot 76100, Israel. Also Department of Mathematics and Department of Mechanical and Aerospace Engineering, University of California, Irvine, California 92697-3875, USA}
\email{etiti@math.uci.edu and edriss.titi@weizmann.ac.il}


%%% ----------------------------------------------------------------------

\begin{abstract}
In this paper, we consider the initial-boundary value problem of the 3D primitive equations for planetary oceanic and atmospheric dynamics with only horizontal eddy viscosity in the horizontal momentum equations and only horizontal diffusion in the temperature equation. Global well-posedness of strong solution is established for any $H^2$ initial data. An $N$-dimensional logarithmic Sobolev embedding inequality, which bounds the $L^\infty$ norm in terms of the $L^q$ norms up to a logarithm of the $L^p$-norm, for $p>N$, of the first order derivatives, and a system version of the classic Gronwall inequality are exploited to establish the required a priori $H^2$ estimates for the global regularity.
\end{abstract}

%%% ----------------------------------------------------------------------
\maketitle
%%% ----------------------------------------------------------------------
%%%%%%%%%%%%%%%%%%%%%%%%%%%%%%%%%%%%%%%%%%%%%%%%%%%%%%%%%%%%%%%%%%%%%%%%%%
%INTRODUCTION%%%%%%%%%%%%%%%%%%%%%%%%%%%%%%%%%%%%%%%%%%%%%%%%%%%%%%%%%%%%%
%%%%%%%%%%%%%%%%%%%%%%%%%%%%%%%%%%%%%%%%%%%%%%%%%%%%%%%%%%%%%%%%%%%%%%%%%%
\allowdisplaybreaks
\section{Introduction}
\label{sec1}

The primitive equations are derived from the Boussinisq system of incompressible flow and they form a fundamental block in models for planetary oceanic and atmospheric dynamics, see, e.g., Lewandowski \cite{LEWAN}, Majda \cite{MAJDA}, Pedlosky \cite{PED}, Vallis \cite{VALLIS}, and Washington and Parkinson \cite{WP}. Due to their importance, the primitive equations has been studied analytically by many authors, see, e.g., \cite{LTW92A,LTW92B,TZ04,PTZ09,MAJDA} and the references therein.

In this paper, we consider the following version primitive equations with only horizontal eddy
viscosities and only horizontal diffusion due to strong dominant horizontal turbulence mixing:
\begin{eqnarray}
&\partial_tv+(v\cdot\nabla_H)v+w\partial_zv+\nabla_Hp-\Delta_H v+f_0k\times
v=0,\label{1.1-1}\\
&\partial_zp+T=0,\label{1.2-1}\\
&\nabla_H\cdot v+\partial_zw=0,\label{1.3-1}\\
&\partial_tT+v\cdot\nabla_H T+w\partial_zT-\Delta_HT=0,\label{1.4-1}
\end{eqnarray}
where the horizontal velocity $v=(v^1,v^2)$, the vertical velocity $w$, the temperature
$T$ and the pressure $p$ are the unknowns, and $f_0$ is the Coriolis parameter. In this paper, we use the notations $\nabla_H=(\partial_x,\partial_y)$
and $\Delta_H=\partial_x^2+\partial_y^2$ to denote the horizontal gradient and the
horizontal Laplacian, respectively.

For the primitive equations with full viscosities and full diffusion, the mathematical analysis was initialed in 1990s by Lions, Temam and Wang \cite{LTW92A,LTW92B,LTW95}, where among other issues they established the global
existence of weak solutions. The uniqueness of weak solutions for 2D case was later proved
by Bresch, Guill\'en-Gonz\'alez, Masmoudi and Rodr\'iguez-Bellido \cite{BGMR03}; however,
the uniqueness of weak solutions for the three-dimensional case is still unclear. Local well-posedness of strong solutions was obtained by Guill\'en-Gonz\'alez, Masmoudi and Rodr\'iguez-Bellido \cite{GMR01}. Global existence
of strong solutions for 2D case was established by Bresch, Kazhikhov and Lemoine in
\cite{BKL04} and Temam and Ziane in \cite{TZ04}, while the 3D case was established in \cite{CAOTITI2}. Global strong solutions for 3D case were also obtained by
Kobelkov \cite{KOB06} later by using a different approach, see also
the subsequent articles by Kukavica and Ziane \cite{KZ07A,KZ07B}.

The systems considered in all the papers \cite{CAOTITI2,LTW92A,LTW92B,LTW95,TZ04,PTZ09,BGMR03,GMR01,BKL04,KOB06,KZ07A,KZ07B} are assumed to have full dissipation, i.e. with both full viscosities and full diffusion. Both physically and mathematically, it is also important and interesting to study
the system with partial dissipation, i.e. with only partial viscosities or only partial diffusion. The first result in this direction for the primitive equations was obtained in \cite{CAOTITI3}, where the authors considered the system with full viscosities but only vertical diffusion, and proved that such a system has a unique global strong solution, provided the local in time one exists. As the complement and a generalization of \cite{CAOTITI3}, the local and global well-posedness of strong solutions are recently established in \cite{CAOLITITI1}, with $H^2$ initial data. As the counterpart of \cite{CAOLITITI1}, global well-posedness of strong solutions to the primitive equations with full viscosities but only horizontal diffusion is later obtained in \cite{CAOLITITI2}, still for $H^2$ initial data. Notably, smooth solutions to the inviscid primitive equation, with or without coupling to the temperature equation, has been shown \cite{CINT} to blow up in finite time (see also Wong \cite{TKW}).

Note that in all the papers \cite{CAOTITI2,LTW92A,LTW92B,LTW95,TZ04,PTZ09,BGMR03,GMR01,BKL04,KOB06,KZ07A,KZ07B,CAOTITI3,
CAOLITITI1,CAOLITITI2}, no matter whether the systems are considered to have full of partial diffusion in the temperature equation, they are assumed to have full viscosities in the horizontal momentum equations. Physically, in the oceanic and atmospheric dynamics, the horizontal scales are much lager than the vertical one with dominant strong horizontal turbulence mixing that induces horizontal viscosities, i.e. system (\ref{1.1-1})--(\ref{1.4-1}). From the mathematical point of view, there are two obvious difficulties in studying system (\ref{1.1-1})--(\ref{1.4-1}). One is that the strongest nonlinear term, i.e. $\left(\int_{-h}^z\nabla_H\cdot vd\xi\right)\partial_zv$, is quadratic in the first derivatives of the unknowns. This is caused by the lack of the dynamical equation for the vertical component of the velocity. The other one is that, due to the lack of the vertical viscosity in the horizontal momentum equations, one can not expect any smoothing effect in the vertical direction.

The aim of this paper is to show that strong solutions exist globally for system (\ref{1.1-1})--(\ref{1.4-1}), subject to some initial and boundary conditions, for any $H^2$ initial data.
More precisely, we consider the problem in the domain $\Omega_0=M\times(-h,0)$, with $M=(0,1)\times(0,1)$, and supplement system (\ref{1.1-1})--(\ref{1.4-1}) with the following boundary and initial conditions:
\begin{eqnarray}
& v, w\mbox{ and } T \mbox{ are }\mbox{periodic in }x \mbox{ and }y,\label{1.5-1}\\
&(\partial_zv,w)|_{z=-h,0}=(0,0),\quad T|_{z=-h}=1,\quad T|_{z=0}=0,\label{1.6-1}\\
&(v,T)|_{t=0}=(v_0, T_0). \label{1.7-1}
\end{eqnarray}

Replacing $T$ and $p$ by $T+\frac{z}{h}$ and $p-\frac{z^2}{2h}$, respectively, then system (\ref{1.1-1})--(\ref{1.4-1}) with (\ref{1.5-1})--(\ref{1.7-1}) is reduced to
\begin{eqnarray}
&\partial_tv+(v\cdot\nabla_H)v+w\partial_zv+\nabla_Hp-\Delta_H v+f_0k\times v=0,\label{1-1.1}\\
&\partial_zp+T=0,\label{1-1.2}\\
&\nabla_H\cdot v+\partial_zw=0,\label{1-1.3}\\
&\partial_tT+v\cdot\nabla_HT+w\left(\partial_zT+\frac{1}{h}\right)-\Delta_HT=0,\label{1-1.4}
\end{eqnarray}
subject to the boundary and initial conditions
\begin{eqnarray}
& v, w, T \mbox{ are }\mbox{periodic in }x \mbox{ and }y ,\label{1-1.5}\\
&(\partial_zv,w)|_{z=-h,0}=0,\quad T|_{z=-h,0}=0,\label{1-1.6}\\
&(v,T)|_{t=0}=(v_0, T_0). \label{1-1.7}
\end{eqnarray}
Here, for simplicity, we still use $T_0$ to denote the initial temperature in (\ref{1-1.7}), though it is obtained by replacing the $T_0$ in (\ref{1.7-1}) by $T_0-\frac{z}{h}$.

Due to the same reasons to those explained in \cite{CAOLITITI1,CAOLITITI2}, system (\ref{1-1.1})--(\ref{1-1.7}) defined on $\Omega_0$ is equivalent to the following system defined on $\Omega:=M\times(-h,h)$:
\begin{eqnarray}
&\partial_tv+(v\cdot\nabla_H)v+w\partial_zv+\nabla_Hp-\Delta_H v+f_0k\times v=0,\label{1.1}\\
&\partial_zp+T=0,\label{1.2}\\
&\nabla_H\cdot v+\partial_zw=0,\label{1.3}\\
&\partial_tT+v\cdot\nabla_HT+w\left(\partial_zT+\frac{1}{h}\right)-\Delta_HT=0,\label{1.4}
\end{eqnarray}
subject to the boundary and initial conditions
\begin{eqnarray}
& v, w, p \mbox{ and } T \mbox{ are }\mbox{periodic in }x, y, z,\label{1.5}\\
& v\mbox{ and }p \mbox{ are even in }z,\mbox{ and } w\mbox{ and }T\mbox{ are odd in }z,\label{1.6}\\
&(v,T)|_{t=0}=(v_0, T_0). \label{1.7}
\end{eqnarray}
Not that the restriction on the sub-domain $\Omega_0$ of a solution $(v,w,p,T)$ to system (\ref{1.1})--(\ref{1.7}) is a solution to the original system (\ref{1-1.1})--(\ref{1-1.7}).
Because of this, throughout this paper, we mainly concern on the study of system (\ref{1.1})--(\ref{1.7}) defined on $\Omega$, while the well-posedness results for system (\ref{1-1.1})--(\ref{1-1.7}) defined on $\Omega_0$ follow as a corollary of those for system (\ref{1.1})--(\ref{1.7}).

One can check that system (\ref{1.1})--(\ref{1.7})
is equivalent to (see \cite{CAOTITI3} for example)
\begin{eqnarray}
&\partial_tv-\Delta_H v+(v\cdot\nabla_H)v-\left(\int_{-h}^z\nabla_H\cdot v(x,y,\xi,t)d\xi\right)\partial_zv\nonumber\\
&\quad\qquad+f_0k\times v+\nabla_H\left(p_s(x,y,t)-\int_{-h}^zT(x,y,\xi,t)d\xi\right)=0,\label{1.8}\\
&\nabla_H\cdot\int_{-h}^hv(x,y,z,t)dz=0,\label{1.9}\\
&\partial_tT-\Delta_HT+v\cdot\nabla_HT-\left(\int_{-h}^z\nabla_H\cdot v(x,y,\xi,t)d\xi\right)\left(\partial_zT+\frac{1}{h}\right)=0,\label{1.10}
\end{eqnarray}
subject to the following boundary and initial conditions
\begin{eqnarray}
&v\mbox{ and } T \mbox{ are periodic in }x, y, z,\label{BC1}\\
&v\mbox{ and } T \mbox{ are even and odd in }z,\mbox{ respectively},\label{BC2}\\
&(v,T)|_{t=0}=(v_0, T_0). \label{IC}
\end{eqnarray}

Before stating our main results, let's introduce some necessary notations and give the definitions of strong solutions. Throughout this paper, for $1\leq q\leq\infty$, we use $L^q(\Omega), L^q(M)$ and $W^{m,q}(\Omega), W^{m,q}(M)$ to denote the standard Lebesgue and Sobolev spaces, respectively. For $q=2$, we use $H^m$ instead of $W^{m,2}$. We use $W_{\text{per}}^{m,q}(\Omega)$ and $H^m_{\text{per}}$ to denote the spaces of periodic functions in $W^{m,q}(\Omega)$ and $H^m(\Omega)$, respectively. For simplicity, we still use the notations $L^p$ and $H^m$ to denote the $N$ product spaces $(L^p)^N$ and $(H^m)^N$, respectively. We always use $\|u\|_p$ to denote the $L^p$ norm of $u$. Moreover, for convenience, we often use $\|(f_1,\cdots,f_n)\|_2^2$ to denote the summation $\sum_{i=1}^n\|f_i\|_2^2$.

\begin{definition}\label{def1.1}
Given a positive time $\mathcal T$, and let $v_0\in H^2(\Omega)$ and $T_0\in H^2(\Omega)$ be two periodic functions, such that they are even and odd in $z$, respectively. A couple $(v,T)$ is called a strong solution to system
(\ref{1.8})-(\ref{IC}) (or equivalently (\ref{1.1})--(\ref{1.7})) on $\Omega\times(0,\mathcal T)$ if

(i) $v$ and $T$ are periodic in $x,y,z$, and they are even and odd in $z$, respectively;

(ii) $v$ and $T$ have the regularities
\begin{eqnarray*}
&(v,T)\in L^\infty(0,\mathcal T; H^2(\Omega))\cap C([0,\mathcal T];H^1(\Omega)),\\
&(\nabla_Hv,\nabla_HT)\in L^2(0,\mathcal T; H^2(\Omega)), \quad(\partial_tv,\partial_tT)\in L^2(0,\mathcal T; H^1(\Omega));
\end{eqnarray*}

(iii) $v$ and $T$ satisfy equations (\ref{1.8})--(\ref{1.10}) a.e. in $\Omega\times(0,\mathcal T)$ and the initial condition (\ref{IC}).
\end{definition}

\begin{definition}
A couple $(v,T)$ is called a global strong solution to system (\ref{1.8})--(\ref{IC}) if it is a strong solution on $\Omega\times(0,\mathcal T)$ for any $\mathcal T\in(0,\infty)$.
\end{definition}

The main result of this paper is the following global well-posedness result.

\begin{theorem}\label{thm1}
Suppose that the periodic functions $v_0,T_0\in H^2(\Omega)$ are even and odd in $z$, respectively. Then system (\ref{1.8})-(\ref{IC}) (or equivalently (\ref{1.1})--(\ref{1.7})) has a unique global strong solution $(v,T)$, which is continuously depending on the initial data.
\end{theorem}

The key issue of proving Theorem \ref{thm1} is establishing the a priori $H^2$ estimates on the strong solutions. Our analysis shows that once the $L^2$ estimate on $u=\partial_zv$ is obtained, all the required estimates of the other derivatives can be successfully achieved. Unfortunately, due to the lack of the vertical viscosity in the horizontal momentum equations, such $L^2$ estimate can not be obtained solely without the contribution of the other derivative of the velocity. We observe that in all the arguments of existing articles the full viscosities play essential role in obtaining the $L^2$ estimate on $u=\partial_zv$, and thus the existing arguments can not be applied in our case. Still caused by the lack of the vertical viscosity, one will encounter $\|v\|_\infty^2$ which appears as the coefficients in the higher order energy inequalities, in other words, the energy inequalities arrive to are all of the form
\begin{equation}
  \label{E1}\frac{d}{dt}f\leq C\|v\|_\infty^2f+\cdots,
\end{equation}
see section \ref{sec4} for the details. Since we do not know whether $\|v\|_\infty^2$ is integrable in time, we can not obtain the required estimate for $f$ directly from such kind of energy inequalities. Besides, recalling that $f$ will represent quantities involving the $H^2$ norm of $v$, though $\|v\|_\infty^2$ can be bounded by $f$, by the Sobolev embedding inequality, one will still be unable to obtain the global in time $H^2$ estimate for $v$ by the simple application of standard embedding inequality. Observe that if $\|v\|_\infty^2$ and $f$ have the relationship
$$
\|v\|_\infty^2\leq C\log f,
$$
then the previous energy inequality (\ref{E1}) implies the global in time estimate for $f$. To guarantee this relationship, thanks to the logarithmic Sobolev embedding inequality (Lemma \ref{log}, below), it suffices to prove that the $L^q$ norms of $v$ grow no faster than $C\sqrt q$. By taking advantage of the property that the pressure $p$ depends essentially only on the horizontal spatial variables, and using the Ladyzhenskaya type inequalities (Lemma \ref{lad}, below) for a class of integrals in 3D, one can successfully prove the desired growth of the $L^q$ norms of $v$, and thus obtain the a priori $H^2$ estimates, and hence the global regularity.

It should be pointed out that, due to the anisotropic structure of the momentum equation (\ref{1.8}) (the advection term $(v\cdot\nabla_H)v$ and the vertical advection term
$\left(\int_{-h}^z\nabla_H\cdot vd\xi\right)\partial_zv$ play different roles), the treatments for different derivatives of the same order will vary. More precisely, when dealing with the derivatives of the same order, the treatment of the vertical derivatives precedes that of the horizontal ones, because the estimates of the horizontal derivatives may depend on those of the vertical ones, see Proposition \ref{prop4.1} and Proposition \ref{prop4.2}, below, for the details. Accordingly, a system version of the classic Gronwall inequality, Lemma \ref{gronwall} below, is exploited to derive the a priori bounds from the energy inequalities.
We believe that this system version of the Gronwall inequality is interesting on its own, and in fact it can benefit us when using the energy approach, see Remark \ref{remark}, below.

As a corollary of Theorem \ref{thm1}, we have the following theorem, which states the global well-posedness of strong solutions to system (\ref{1-1.1})--(\ref{1-1.7}). Strong solutions to system (\ref{1-1.1})--(\ref{1-1.7}) are defined in the similar way as before.

\begin{theorem}
  \label{thm2}
Let $v_0$ and $T_0$ be two functions defined on $\Omega_0$, such that they are both periodic in $x$ and $y$. Denote by $v_0^{\text{ext}}$ and $T_0^{\text{ext}}$ the even and odd extensions in $z$ of $v_0$ and $T_0$, respectively. Suppose that $v_0^{\text{ext}}, T_0^{\text{ext}}\in H^1_{\text{per}}(\Omega)$. Then system (\ref{1-1.1})--(\ref{1-1.7}) has a unique local strong solution $(v,T)$.
\end{theorem}

Theorem \ref{thm2} follows directly by applying Theorem \ref{thm1} with initial data $(v_0^{\text{ext}}, T_0^{\text{ext}})$ and restricting the solution on the sub-domain $\Omega_0$.

The rest of this paper is arranged as follows: in the next section, section \ref{sec2}, we collect some preliminary results which will be used in the subsequent sections; in section \ref{sec3}, we establish the a priori low order energy estimates, which are independent of the regularization parameter $\varepsilon$, for strong solutions to a regularized system, while the $\varepsilon$ independent higher order energy inequalities are given in section \ref{sec4}. In section \ref{sec5}, by the aid the a priori estimates and the higher order energy inequalities achieved in the previous two sections, we first establish the a priori $H^2$ estimates for the strong solutions to the regularized system, and then obtain the global well-posedness of strong solutions to system (\ref{1.8})-(\ref{IC}) (or equivalently system (\ref{1.1})--(\ref{1.7})) by standard approach. Section \ref{sec6} is an appendix in which an $N$-dimensional logarithmic Sobolev embedding inequality is established.

Throughout this paper, $C$ denotes a general constant which may be different from line to line.

\section{Preliminaries}
\label{sec2}

In this section we collect some preliminary results which will be used in the rest of this paper, and we start with the following Ladyzhenskaya type inequality in 3D for a class of integrals, which will be frequently used throughout the paper.

\begin{lemma} \label{lad}
The following inequalities hold true
\begin{align*}
&\int_M\left(\int_{-h}^h|\phi(x,y,z)|dz\right)\left(\int_{-h}^h|\varphi(x,y,z)\psi(x,y,z)|dz\right)dxdy\\
\leq&C\min\left\{\|\phi\|_2^{\frac{1}{2}}\left(\|\phi\|_2^{\frac{1}{2}}
+\|\nabla_H\phi\|_{2}^{\frac{1}{2}
}\right)\|\varphi\|_2\|\psi\|_2^{\frac{1}{2}}\left(
\|\psi\|_2^{\frac{1}{2}}+\|\nabla_H\psi\|_2^{\frac{1}{2}}\right)\right.,\\
&\left.\|\phi\|_2\|\varphi\|_2^{\frac{1}{2}}\left(\|\varphi\|_2^{\frac{1}{2}}+\|\nabla_H\varphi\|_{2}^{\frac{1}{2}}
\right)\|\psi\|_2^{\frac{1}{2}}\left(
\|\psi\|_2^{\frac{1}{2}}+\|\nabla_H\psi\|_2^{\frac{1}{2}}\right)\right\},
\end{align*}
and
\begin{align*}
&\int_M\left(\int_{-h}^h|\phi(x,y,z)|dz\right)\left(\int_{-h}^h|\varphi(x,y,z)
\nabla_H\Psi(x,y,z)|dz\right)dxdy\\
\leq&C\min\left\{\|\phi\|_2^{\frac{1}{2}}\left(\|\phi\|_2^{\frac{1}{2}}
+\|\nabla_H\phi\|_{2}^{\frac{1}{2}
}\right)\|\varphi\|_2\|\Psi\|_\infty^{\frac{1}{2}}\|\nabla_H^2\Psi\|_2^{\frac{1}{2}},\right.\\
&\left.\|\phi\|_2\|\varphi\|_2^{\frac{1}{2}}\left(\|\varphi\|_2^{\frac{1}{2}}+\|\nabla_H\varphi\|_{2}^{\frac{1}{2}}
\right)\|\Psi\|_\infty^{\frac{1}{2}}\|\nabla_H^2\Psi\|_2^{\frac{1}{2}}\right\},
\end{align*}
for every $\phi,\varphi,\psi,\Psi$ such that the right hand sides make sense and are finite. Moreover, if $\phi$ has the form $\phi=\nabla_Hf$, for some function $f$, then by the Poinc\'are inequality, the lower order term $\|\phi\|_2^{\frac{1}{2}}$ in the parentheses can be dropped in the above inequalities, and the same can also be said for $\varphi$ and $\psi$.
\end{lemma}

\begin{proof}
Similar inequalities has been established in \cite{CAOTITI1} and \cite{CAOTITI3}. However, for completeness, we sketch the proofs here, and the ideas used here are the same as in those papers. By the H\"older and Minkowski inequalities, we deduce
\begin{align*}
  &\int_M\left(\int_{-h}^h|\phi(x,y,z)| dz\right)\left(\int_{-h}^h|\varphi(x,y,z)\psi(x,y,z)| dz\right)dxdy\\
  \leq&\int_M\left(\int_{-h}^h|\phi(x,y,z)|dz\right)\left(\int_{-h}^h|\varphi(x,y,z)
  |^2dz\right)^{\frac{1}{2}}
  \left(\int_{-h}^h|\psi(x,y,z)|^2dz\right)^{\frac{1}{2}}dxdy\\
  \leq&\min\left\{\left(\int_M\left|\int_{-h}^h|\phi|dz\right|^2dxdy\right)^{\frac{1}{2}}
  \left(\int_M\left|\int_{-h}^h|\varphi|^2dz\right|^2dxdy\right)^{\frac{1}{4}},\right.\\
  &\left.\left(\int_M\left|\int_{-h}^h|\phi|dz\right|^4dxdy\right)^{\frac{1}{4}}
  \left(\int_M\int_{-h}^h|\varphi|^2dzdxdy\right)^{\frac{1}{2}}\right\} \left(\int_M\left|\int_{-h}^h|\psi|^2dz\right|^2dxdy\right)^{\frac{1}{4}}\\
  \leq&C\min\left\{\|\phi\|_2\left(\int_{-h}^h\|\varphi\|_{4,M}^2dz\right)^{\frac{1}{2}},
  \int_{-h}^h\|\phi\|_{4,M}dz\|\varphi\|_2
  \right\}\left(\int_{-h}^h\|\psi\|_{4,M}^2dz\right)^{\frac{1}{2}}.
\end{align*}
Similarly
\begin{align*}
  &\int_M\left(\int_{-h}^h|\phi(x,y,z)| dz\right)\left(\int_{-h}^h|\varphi(x,y,z)\nabla_H\Psi(x,y,z)| dz\right)dxdy\\
  \leq&C\min\left\{\|\phi\|_2\left(\int_{-h}^h\|\varphi\|_{4,M}^2dz\right)^{\frac{1}{2}},
  \int_{-h}^h\|\phi\|_{4,M} dz \|\varphi\|_2
  \right\}\left(\int_{-h}^h\|\nabla_H\Psi\|_{4,M}^2dz\right)^{\frac{1}{2}}.
\end{align*}
It follows from the 2D Ladyzhenskaya and Gagliardo-Nirenberg inequalities that
\begin{align*}
&\int_{-h}^h\|\phi\|_{4,M}dz\leq C\int_{-h}^h\|\phi\|_{2,M}^{\frac{1}{2}}\|\phi\|_{H^1(M)}^{\frac{1}{2}}dz\leq C\|\phi\|_2^{\frac{1}{2}}\left(\|\phi\|_2^{\frac{1}{2}}+\|\nabla_H\phi\|_2^{\frac{1}{2}}\right),\\
  &\int_{-h}^h\|\psi\|_{4,M}^2dz\leq C\int_{-h}^h\|\psi\|_{2,M}\|\psi\|_{H^1(M)}dz\leq C\|\psi\|_2(\|\psi\|_2+\|\nabla_H\psi\|_2),\\
  &\int_{-h}^h|\nabla_H\Psi\|_{4,M}dz\leq C\int_{-h}^h\|\Psi\|_{\infty,M}\|\Psi\|_{H^2(M)}dz\leq C\|\Psi\|_\infty\|\nabla_H^2\Psi\|_2.
\end{align*}
The conclusions follow from combining the previous five inequalities.
\end{proof}

The following logarithmic Sobolev inequality, which bounds the $L^\infty$ norm in terms of the $L^q$ norms up to the logarithm of the norms of the high order derivatives, will play an important role in establishing the a priori $H^2$ estimates later. Some relevant inequalities can be found in \cite{CAOWU,CAOFARHATTITI,DANCHINPAICU}, where the two-dimensional case is considered.

\begin{lemma}
  \label{log}
Let $F\in W^{1,p}(\Omega)$, with $p>3$, be a periodic function. Then the following inequality holds true
\begin{equation*}
  \|F\|_\infty\leq C_{ {p},\lambda}\max\left\{1,\sup_{r\geq2}\frac{\|F\|_r}{r^\lambda}\right\}
  \log^\lambda
  (\|F\|_{W^{1, {p}}(\Omega)}+e),
\end{equation*}
for any $\lambda>0$.
\end{lemma}

\begin{proof}
Extending $F$ periodically to the whole space. Take a function $\phi\in C_0^\infty(\mathbb R^3)$, such that $\phi\equiv1$ on $\Omega$, and $0\leq\phi\leq1$ on $\mathbb R^3$. Set $f=F\phi$. By Lemma \ref{lemapp} (choose $R=1$ there) in the appendix, it holds that
\begin{equation*}
  \|f\|_{L^\infty(\mathbb R^3)}\leq C_{ {p},\lambda}\max\left\{1,\sup_{r\geq2}\frac{\|f\|_{L^r(\mathbb R^3)}}{r^\lambda}\right\}
  \log^\lambda
  (\|f\|_{W^{1, {p}}(\mathbb R^3)}+e).
\end{equation*}
Noticing that
$$
\|F\|_\infty\leq\|f\|_{L^\infty(\mathbb R^3)}, \quad\|f\|_{L^r(\mathbb R^3)}\leq C\|F\|_r,\quad\|f\|_{W^{1,p}(\mathbb R^3)}\leq C\|F\|_{W^{1,p}(\Omega)},
$$
we deduce
\begin{align*}
  \|F\|_\infty\leq&\|f\|_{L^\infty(\mathbb R^3)}\leq C_{ {p},\lambda}\max\left\{1,\sup_{r\geq2}\frac{\|f\|_{L^r(\mathbb R^3)}}{r^\lambda}\right\}
  \log^\lambda
  (\|f\|_{W^{1, {p}}(\mathbb R^3)}+e)\\
  \leq&C_{ {p},\lambda}\max\left\{1,\sup_{r\geq2}\frac{\|F\|_r}{r^\lambda}\right\}
  \log^\lambda
  (\|F\|_{W^{1, {p}}(\Omega)}+e),
\end{align*}
proving the conclusion.
\end{proof}

The following lemma is a system version of the classic Gronwall inequality.

\begin{lemma}
  \label{gronwall}
Let $m(t), K(t),
A_i(t)$ and $B_i(t)$ be nonnegative functions, such that $A_i\geq e,$ are absolutely continuous, for $i=1,\cdots,n, K\in L_{\text{loc}}^1([0,\infty))$, and
\begin{equation*}
   m(t)\leq K(t)\log\sum_{i=1}^nA_i(t).
\end{equation*}
Given a positive time $\mathcal T$, and suppose that
\begin{align}
&\frac{d}{dt}A_1(t)+B_1(t)\leq m(t)A_1(t),\label{e0}\\
&\frac{d}{dt}A_i(t)+B_i(t)\leq m(t)A_i(t)+\zeta A_{i-1}^\alpha(t)B_{i-1}(t),\quad i=2,\cdots,n,\label{ei}
\end{align}
for any $t\in(0,\mathcal T)$, where $\alpha\geq1$ and $\zeta\geq1$ are two constants.
Then it holds that
$$
\sum_{i=1}^nA_i(t)+\sum_{i=1}^n\int_0^tB_i(s)ds\leq Q(t),\quad \forall t\in[0,\mathcal T),
$$
where $Q$ is a continuous function on $[0,\infty)$ which is determined by $A_i(0), i=1,\cdots,n,$ and $K$, given explicitly in equation (\ref{E2}) below.
\end{lemma}

\begin{proof}
Multiplying inequality (\ref{e0}) by $\zeta(\alpha+1)A_1^\alpha$ yields
$$
\zeta\frac{d}{dt}A_1^{\alpha+1}+\zeta(\alpha+1)A_1^\alpha B_1\leq \zeta(\alpha+1)mA_1^{\alpha+1}.
$$
Summing this with (\ref{ei}) for $i=2$ leads to
$$
\frac{d}{dt}(A_2+\zeta A_1^{\alpha+1})+\zeta\alpha A_1^\alpha B_1+B_2\leq(\alpha+1)m(A_2+\zeta A_1^{\alpha+1}).
$$
Set $\mathcal A_1=A_1, \mathcal A_2=A_2+\zeta\mathcal A_1^{\alpha+1}, \mathcal B_1=B_1$ and $\mathcal B_2=\mathcal B_1+B_2$, then the above inequality gives
\begin{equation*}
\frac{d}{dt}\mathcal A_2+\mathcal B_2\leq(\alpha+1)m\mathcal A_2.
\end{equation*}
Multiplying the above inequality by $\zeta(\alpha+1)\mathcal A_2^\alpha$ yields
$$
\zeta\frac{d}{dt}\mathcal A_2^{\alpha+1}+\zeta(\alpha+1)\mathcal A_2^\alpha\mathcal B_2\leq \zeta(\alpha+1)^2m\mathcal A_2^{\alpha+1}.
$$
Summing up this inequality with ($\ref{ei}$) for $i=3$ implies
$$
\frac{d}{dt}(A_3+\zeta\mathcal A_2^{\alpha+1})+\zeta\alpha\mathcal A_2^\alpha\mathcal B_2+B_3\leq(\alpha+1)^2m(A_3+\zeta\mathcal A_2^{\alpha+1}).
$$
Set $\mathcal A_3=A_3+\zeta\mathcal A_2^{\alpha+1}$ and $\mathcal B_3=B_3+\mathcal B_2$, then we have
\begin{equation*}
\frac{d}{dt}\mathcal A_3+\mathcal B_3\leq(\alpha+1)^2m\mathcal A_3.
\end{equation*}

Continuing the previous procedure inductively, we obtain
\begin{eqnarray*}
&\mathcal A_i=A_i+\zeta\mathcal A_{i-1}^{\alpha+1},\quad\mathcal B_i=B_i+\mathcal B_{i-1}, \quad i=2,\cdots,N,\\
&\frac{d}{dt}\mathcal A_i+\mathcal B_i\leq (\alpha+1)^{i-1}m\mathcal A_i,\quad i=1,2,\cdots,N,
\end{eqnarray*}
in particular, it holds that
\begin{equation}
\frac{d}{dt}\mathcal A_N+\mathcal B_N\leq(\alpha+1)^{N-1}m\mathcal A_N.\label{EN}
\end{equation}
By the assumption on $m(t)$, the above inequality implies
\begin{equation*}
\frac{d}{dt}\mathcal A_{N}(t)\leq (\alpha+1)^{N-1}K(t)\mathcal A_{N}(t)\log\mathcal A_{N}(t).
\end{equation*}
Therefore
\begin{equation*}
\frac{d}{dt}\log\mathcal A_{N}(t)\leq (\alpha+1)^{N-1}K(t)\log\mathcal A_{N},
\end{equation*}
from which, we obtain
\begin{equation*}
 \log\mathcal A_{N}\leq e^{(\alpha+1)^{N-1}\int_0^tK(s)ds}\log\mathcal A_N(0)=:q_0(t),
\end{equation*}
and
$$
\mathcal A_{N}(t)\leq e^{q_0(t)}=:q_1(t),
$$
for $t\in[0,\mathcal T)$.
Note that $q_1$ is an increasing function on $[0,\infty)$. Thanks to this estimate, it follows from integrating inequality (\ref{EN}) with respect to $t$ that
\begin{align*}
  &\mathcal A_N(t)+\int_0^t\mathcal B_N(s)ds\leq(\alpha+1)^{N-1}\int_0^tm(s)\mathcal A_N(s)ds\\
  \leq&(\alpha+1)^{N-1}\int_0^t\mathcal A_N(s)K(s)\log\mathcal A_N(s)ds\\
  \leq&(\alpha+1)^{N-1}\int_0^tK(s)dsq_1(t)q_0(t)=Q(t),
\end{align*}
for all $t\in[0,\mathcal T)$, where
\begin{equation}
  Q(t)=(\alpha+1)^{N-1}\int_0^tK(s)dsq_1(t)q_0(t). \label{E2}
\end{equation}
From which
one obtains the conclusion.
\end{proof}

\begin{remark}
\label{remark}
Lemma \ref{gronwall} indicates that when doing the energy estimates, step by step,
the quantities appear on the
left-hand side in the previous steps can be treated freely as if they were a priori bounded, provided the coefficient term does not grow too fast compared to the quantities under consideration (no faster than the logarithm of the summation of them). This gives us a large room to handle some hard terms in the current step.
\end{remark}

We also need the following Aubin-Lions lemma.

\begin{lemma}\label{AL}
(Aubin-Lions Lemma, See Simon \cite{Simon} Corollary 4) Assume that $X, B$ and $Y$ are three Banach spaces, with $X\hookrightarrow\hookrightarrow B\hookrightarrow Y.$ Then it holds that

(i) If $F$ is a bounded subset of $L^p(0, T; X)$, where $1\leq p<\infty$, and $\frac{\partial F}{\partial t}=\left\{\frac{\partial f}{\partial t}|f\in F\right\}$ is bounded in $L^1(0, T; Y)$, then $F$ is relatively compact in $L^p(0, T; B)$.

(ii) If $F$ is bounded in $L^\infty(0, T; X)$ and $\frac{\partial F}{\partial t}$ is bounded in $L^r(0, T; Y)$, where $r>1$, then $F$ is relatively compact in $C([0, T]; B)$.
\end{lemma}

Finally, we will use the following global existence result for a regularized system.

\begin{proposition}
  \label{lem2.4}
Suppose that the periodic functions $v_0,T_0\in H^2(\Omega)$ are even and odd in $z$, respectively. Then for any $\varepsilon>0$, there is a unique global strong solution $(v,T)$ to the following system
\begin{eqnarray}
    &\partial_tv+(v\cdot\nabla_H)v-\left(\int_{-h}^z\nabla_H\cdot v(x,y,\xi,t)d\xi\right)\partial_zv-\Delta_Hv-\varepsilon\partial_z^2v  \nonumber\\
    &+f_0k\times v+\nabla_H\left(p_s(x,y,t)-\int_{-h}^zT(x,y,\xi,t)d\xi\right)=0,\label{eq1}\\
    &\int_{-h}^h\nabla_H\cdot v(x,y,z,t)dz=0,\label{eq2}\\
    &\partial_tT+v\cdot\nabla_HT-\left(\int_{-h}^z\nabla_H\cdot v(x,y,\xi,t)d\xi\right)
\left(\partial_zT+\frac{1}{h}\right)-\Delta_H T-\varepsilon\partial_z^2T=0, \label{eq3}
\end{eqnarray}
subject to the boundary and initial conditions (\ref{BC1})--(\ref{IC}), such that
\begin{align*}
  &(v,T)\in L^\infty_{\text{loc}}([0,\infty);H^2(\Omega))\cap C([0,\infty);H^1(\Omega))\cap
  L^2_{\text{loc}}([0,\infty);H^3(\Omega)),\\
  &(\partial_tv,\partial_tT)\in L^2_{\text{loc}}([0,\infty);H^1(\Omega)).
\end{align*}
\end{proposition}

\begin{proof}
The proof can be given in the same way as in \cite{CAOLITITI1} (see Proposition 2.1 there), and thus we omit it here.
\end{proof}

\section{Low order energy estimates}
\label{sec3}

In this section, we work on the low order energy estimates on the strong solution to system (\ref{eq1})--(\ref{eq3}), subject to the boundary and initial conditions (\ref{BC1})--(\ref{IC}). In particular, we prove that the growth of the $L^q$ norms of $v$ is not faster than $C\sqrt q$, for a constant $C$ independent of $q$.

\begin{proposition}\label{prop3.1}
Let $(v,T)$ be the global strong solution to system (\ref{eq1})--(\ref{eq3}), subject to the boundary and initial conditions (\ref{BC1})--(\ref{IC}). Then for any $\mathcal T\in(0,\infty)$, we have the following:

(i) basic energy estimates
$$
\sup_{0\leq t\leq\mathcal T}\|(v,T)\|_2^2+\int_0^{\mathcal T}\|(\nabla_Hv, \nabla_HT, \sqrt\varepsilon\partial_zv,
\sqrt\varepsilon\partial_zT)\|_2^2dt\leq  K_0(\mathcal T),
$$
where $K_0(\mathcal T)=Ce^\mathcal {T}(\|v_0\|_2^2+\|T_0\|_2^2)$, for a positive constant $C$ depending only on $h$;

(ii) $L^q$ estimate on  $T$, with $2\leq q\leq\infty$
$$
\sup_{0\leq t\leq \mathcal T}\|T^*\|_q\leq   \|T^*_0\|_q,
$$
where $T^*=T+\frac{z}{h}$ and $T^*_0=T_0+\frac{z}{h}$;

(iii) $L^q$ estimate on $v$, with $q\in[4,\infty)$,
$$
\sup_{0\leq t\leq\mathcal T}\|v\|_q \leq K_1(\mathcal T)e^{C\|T_0\|_q^2\mathcal T}
(1+\|v_0\|_q)\sqrt q,
$$
where $K_1$ is a continuously increasing function determined by $\|v_0\|_2, \|T_0\|_2, \|v_0\|_4 $and $\|T_0\|_4$.
\end{proposition}

\begin{proof}
(i) Multiplying equations (\ref{eq1}) and (\ref{eq3}) by $v$ and $T$, respectively, summing the resulting equations up, and integrating over $\Omega$, then it follows from integrating by parts and using (\ref{eq2}) that
\begin{align*}
  &\frac{1}{2}\frac{d}{dt}\int_\Omega(|v|^2+|T|^2)dxdydz\\
  &+\int_\Omega\Big(|\nabla_Hv|^2+\varepsilon|\partial
  _zv|^2+|\nabla_HT|^2+\varepsilon|\partial
  _zT|^2\Big)dxdydz\\
  =&-\int_\Omega\left[\left(\int_{-h}^zTd\xi\right)\nabla_H\cdot v-\frac{1}{h}\left(\int_{-h}^z\nabla_H\cdot vd\xi\right)T\right]dxdydz\\
  \leq&C\|T\|_2\|\nabla_Hv\|_2\leq\frac{1}{2}\|\nabla_Hv\|_2^2+C\|T\|_2^2.
\end{align*}
Thus
$$
\frac{d}{dt}\|(v,T)\|_2^2+\|(\nabla_Hv,\nabla_HT,\sqrt\varepsilon\partial_zv,\sqrt\varepsilon \partial_zT)\|_2^2)
\leq C\|T\|_2^2,
$$
from which, by the Gronwall inequality, one obtains (i).

(ii) Recalling the definition of $T^*$, using equation (\ref{eq3}), one can easily check that $T^*$ satisfies
$$
\partial_tT^*+v\cdot\nabla_HT^*-\left(\int_{-h}^z\nabla_H\cdot v(x,y,\xi,t)d\xi\right)
\partial_zT^*-\Delta_HT^*-\varepsilon\partial_z^2T^*=0.
$$
Multiplying the above equation by $|T^*|^{q-2}T^*$, with $2\leq q<\infty$, integrating over $\Omega$, then it follows from integration by parts and using the divergence free condition (\ref{eq2}) that
\begin{equation*}
  \frac{1}{q}\frac{d}{dt}\|T^*\|_q^q\leq0,
\end{equation*}
which implies the conclusion for $2\leq q<\infty$. The case that $q=\infty$ follows by taking $q\rightarrow\infty$ and using the fact that $\|T^*\|_q\rightarrow\|T^*\|_\infty$ as $q\rightarrow\infty$.

(iii) Let $4\leq q<\infty$. The $L^\infty(0,\mathcal T; L^q(\Omega))$ estimate on $v$ is proved in two steps: a rough estimate and then a more refined estimate. As we shall see below, the latter is based on the former.

\textbf{Step 1: the rough $L^\infty(0,\mathcal T; L^q(\Omega))$ estimate on $v$.} Multiplying equation (\ref{eq1}) by $|v|^{q-2}v$, and integrating the resulting equation over $\Omega$, then it follows from integrating by parts that
\begin{align}
  \frac{1}{q}\frac{d}{dt}\int_\Omega|v|^qdxdydz+&\int_\Omega|v|^{q-2}\Big(|\nabla_Hv|^2+(q-2)
  \big|\nabla_H|v|\big|^2\nonumber\\
  +&\varepsilon|\partial_zv|^2+(q-2)\varepsilon
  \big|\partial_z|v|\big|^2\Big)dxdydz=I_1+I_2,\label{roughqv1}
\end{align}
where
\begin{eqnarray*}
&&I_1:=\int_\Omega\nabla_H\left(\int_{-h}^zTd\xi\right)\cdot|v|^{q-2}vdxdydz,\\ &&I_2:=-\int_\Omega\nabla_Hp_s(x,y,t)\cdot|v|^{q-2}vdxdydz.
\end{eqnarray*}

Estimate for $I_1$ is given as follows. By the H\"older and Young inequalities, and using (ii), we deduce
\begin{align}
  I_1=&\int_\Omega\nabla_H\left(\int_{-h}^zTd\xi\right)\cdot|v|^{q-2}vdxdydz\nonumber\\
  =&-\int_\Omega\left(\int_{-h}^zTd\xi\right)(|v|^{q-2}\nabla_H\cdot v+(q-2)|v|^{q-3}v\cdot\nabla_H|v|)dxdydz\nonumber\\
  \leq&\int_\Omega\left|\int_{-h}^zTd\xi\right||v|^{q-2}(|\nabla_Hv|+(q-2)\big|\nabla_H|v|\big|) dxdydz\nonumber\\
  \leq&\left\|\int_{-h}^zTd\xi\right\|_q\|v\|_q^{\frac{q}{2}-1} \left(\left\||v|^{\frac{q}{2}-1}\nabla _Hv\right\|_2+(q-2)\left\||v|^{\frac{q}{2}-1}\nabla_H|v|\right\|_2\right)\nonumber\\
  \leq&C\|T\|_q\|v\|_q^{\frac{q}{2}-1}\left(\left\||v|^{\frac{q}{2}-1}\nabla _Hv\right\|_2+(q-2)\left\||v|^{\frac{q}{2}-1}\nabla_H|v|\right\|_2\right)\nonumber\\
  \leq&\frac{1}{8}\left(\left\||v|^{\frac{q}{2}-1}\nabla _Hv\right\|_2^2+(q-2)\left\||v|^{\frac{q}{2}-1}\nabla_H|v|\right\|_2^2\right)+Cq\|T\|_q^2\|v\|_q^
  {q-2}\nonumber\\
  \leq&\frac{1}{8}\left(\left\||v|^{\frac{q}{2}-1}\nabla _Hv\right\|_2^2+(q-2)\left\||v|^{\frac{q}{2}-1}\nabla_H|v|\right\|_2^2\right)+Cq(1+\|T_0\|_q^2)\|v\|_q^
  {q-2},\label{I1}
\end{align}
where the constant $C$ is independent of $q\in[4,\infty)$.

Applying the operator $\frac{1}{2h}\int_{-h}^h\text{div}_H(\cdot)dz$ to equation (\ref{eq1}), and using (\ref{eq2}), it follows from integrating by parts that
\begin{align*}
  -\Delta_Hp_s(x,y,t)=&\text{div}_H\text{div}_H\left(\frac{1}{2h}\int_{-h}^hv\otimes vdz\right)+\text{div}_H\left(\frac{1}{2h}\int_{-h}^hf_0k\times vdz\right)\\
  &-\Delta_H\left(\frac{1}{2h}\int_{-h}^h\left(\int_{-h}^zTd\xi-\int_M\int_{-h}^zTd\xi dxdy\right)dz\right).
\end{align*}
Note that $p_s(x,y,t)$ can be chosen in a unique way by assuming that $\int_Mp_s(x,y,t)dxdy=0$. Set
$$
p_s^0=\frac{1}{2h}\int_{-h}^h\left(\int_{-h}^zTd\xi-\int_M\int_{-h}^zTd\xi dxdy\right)dz,
$$
and decompose $p_s$ as $p_s=p_s^0+p_s^1+p_s^2$, with
\begin{equation*}
  \left\{
  \begin{array}{lr}
    -\Delta_Hp_s^1=\text{div}_H\text{div}_H\left(\frac{1}{2h}\int_{-h}^h(v\otimes v)dz\right),&\mbox{ in }M,\\
    \int_Mp_s^1dxdy=0,&p_s^1\mbox{ is periodic,}
  \end{array}
  \right.
\end{equation*}
and
\begin{equation*}
  \left\{
  \begin{array}{lr}
    -\Delta_Hp_s^2=\text{div}_H\left(\frac{1}{2h}\int_{-h}^hf_0k\times vdz\right),&\mbox{ in }M,\\
    \int_Mp_s^2dxdy=0,&p_s^2\mbox{ is periodic.}
  \end{array}
  \right.
\end{equation*}
Then by the elliptic estimates, we have
\begin{eqnarray}
  &\|p_s^1\|_{q,M}\leq C_q\left\|\int_{-h}^h(v\otimes v)dz\right\|_{q,M}\leq C_q\int_{-h}^h\|v\|_{2q,M}^2dz,\quad\mbox{for all }q\in(1,\infty),\label{p11}\\
  &\|\nabla_Hp_s^1\|_{2,M}\leq C\left\|\text{div}_H\left(\int_{-h}^hv\otimes vdz\right)\right\|_{2,M}\leq C\||v||\nabla_Hv|\|_2,\label{p12}\\
  &\|\nabla_Hp_s^2\|_{2,M}\leq C\left\|\int_{-h}^hk\times vdz\right\|_{2,M}
  \leq C\|v\|_2. \label{p2}
\end{eqnarray}

Setting
$$
I_{2i}:=-\int_\Omega\nabla_Hp_s^i(x,y,t)\cdot|v|^{q-2}vdxdydz,\quad i=0,1,2,
$$
then it is obvious that $I_2=I_{20}+I_{21}+I_{22}$.
We estimate $I_{2i}, i=0,1,2$, as follows. For $I_{20}$, a similar argument as (\ref{I1}) yields
\begin{align}\label{I20}
  I_{20}\leq \frac{1}{8}\left(\left\||v|^{\frac{q}{2}-1}\nabla _Hv\right\|_2^2+(q-2)\left\||v|^{\frac{q}{2}-1}\nabla_H|v|\right\|_2^2\right)+Cq(1+\|T_0\|_q^2)\|v\|_q^
  {q-2},
\end{align}
with constant $C$ independent of $q$. For $I_{22}$, by the H\"older, Minkowski, Ladyzhenskaya and Young inequalities, we deduce
\begin{align}
I_{22}=&-\int_\Omega\nabla_Hp_s^2(x,y,t)\cdot|v|^{q-2}vdxdydz\nonumber\\
  \leq&\int_M|\nabla_Hp_s^2(x,y,t)|\left(\int_{-h}^h|v|^{q-1}dz\right)dxdy\nonumber\\
  \leq&\left(\int_M|\nabla_Hp_s^2(x,y,t)|^2dxdy\right)^{\frac{1}{2}}\left[
  \int_M\left(\int_{-h}^h|v|^{q-1}dz\right)^2dxdy\right]^{\frac{1}{2}}\nonumber\\
  \leq&C\|\nabla_Hp_s^2\|_{2,M}\int_{-h}^h\|v\|_{2(q-1),M}^{q-1}dz
  \leq C\|\nabla_Hp_s^2\|_{2,M}\int_{-h}^h\|v\|_{2,M}^{\frac{1}{q-1}}\|v\|_{2q,M}^{\frac{q(q-2)}{q-1}}dz\nonumber\\
  =&C\|\nabla_Hp_s^2\|_{2,M}\int_{-h}^h\|v\|_{2,M}^{\frac{1}{q-1}}\left\||v|^{\frac{q}{2}}
  \right\|_{4,M}^{\frac{2(q-2)}{q-1}}dz\nonumber\\
  \leq&C\|\nabla_Hp_s^2\|_{2,M}\int_{-h}^h\|v\|_{2,M}^{\frac{1}{q-1}}\left\||v|^{\frac{q}{2}}
  \right\|_{2,M}^{\frac{q-2}{q-1}}\left\||v|^{\frac{q}{2}}\right\|_{H^1(M)}^{\frac{q-2}{q-1}}dz\nonumber\\
  \leq&C\|\nabla_Hp_s^2\|_{2,M}\left(\int_{-h}^h\|v\|_{2,M}^2dz\right)^{\frac{1}{2(q-1)}}\left(\int_{-h}^h
  \left\||v|^{\frac{q}{2}}\right\|_{2,M}^2dz\right)^{\frac{q-2}{2(q-1)}}\nonumber\\
  &\times\left[\int_{-h}^h\left(\left\||v|^{\frac{q}{2}}\right\|_{2,M}^2+
  \left\|\nabla_H|v|^{\frac{q}{2}}\right\|_{2,M}^2\right)dz\right]^{\frac{q-2}{2(q-1)}}\left(\int_{-h}^h
  1dz\right)^{\frac{1}{2(q-1)}}\nonumber\\
  \leq&C(2h)^{\frac{1}{2(q-1)}}\|\nabla_Hp_s^2\|_{2,M}\|v\|_2^{\frac{1}{q-1}}\left\||v|^{\frac{q}{2}}
  \right\|_2^{\frac{q-2}{q-1}}\left(\left\||v|^{\frac{q}{2}}\right\|_{2}^{\frac{q-2}{q-1}}+
  \left\|\nabla_H|v|^{\frac{q}{2}}\right\|_{2}^{\frac{q-2}{q-1}}\right)\nonumber\\
  \leq&C\|\nabla_Hp_s^2\|_{2,M}\|v\|_2^{\frac{1}{q-1}}\left[\|v\|_q^{\frac{q(q-2)}{q-1}}
  +\left\||v|^{\frac{q}{2}}\right\|_2^{\frac{q-2}{q-1}}\left(\frac{q}{2}\left\|
  |v|^{\frac{q}{2}-1}\nabla_H|v|\right\|_2\right)^{\frac{q-2}{q-1}}\right]\nonumber\\
  \leq&C\|\nabla_Hp_s^2\|_{2,M}\|v\|_2^{\frac{1}{q-1}}\|v\|_q^{\frac{q(q-2)}{q-1}}
  \nonumber\\
  &+C\|\nabla_Hp_s^2\|_{2,M}\|v\|_2^{\frac{1}{q-1}}\left\||v|^{\frac{q}{2}}\right\|_2^{\frac{q-2}{q-1}}\left(\sqrt{\frac{q-2}{8}}\left\|
  |v|^{\frac{q}{2}-1}\nabla_H|v|\right\|_2\right)^{\frac{q-2}{q-1}}q^{\frac{q-2}{2(q-1)}}
  \nonumber\\
  \leq&\frac{q-2}{8}\left\|
  |v|^{\frac{q}{2}-1}\nabla_H|v|\right\|_2^2+C\|\nabla_Hp_s^2\|_{2,M}^{\frac{2(q-1)}{q}}\|v\|_2^{\frac{2}{q}}
  \left\||v|^{\frac{q}{2}}\right\|_2^{\frac{2(q-2)}{q}}q^{\frac{q-2}{q}}\nonumber\\
  &+C\|\nabla_Hp_s^2\|_{2,M}\|v\|_2^{\frac{1}{q-1}}\|v\|_q^{\frac{q(q-2)}{q-1}}\nonumber\\
  \leq&\frac{q-2}{8}\left\|
  |v|^{\frac{q}{2}-1}\nabla_H|v|\right\|_2^2+Cq(1+\|v\|_2^2)(1+\|\nabla_Hp_s^2\|_{2,M}^2)\|v\|_q^{q-2}\nonumber\\
  &+C(1+\|v\|_2^2)(1+\|\nabla_Hp_s^2\|_{2,M}^2)\|v\|_q^{\frac{q(q-2)}{q-1}},\nonumber\\
  \leq&C(1+\|v\|_2^2)(1+\|\nabla_Hp_s^2\|_{2,M}^2)\left(q\|v\|_q^{q-2}+
  \|v\|_q^{\frac{q(q-2)}{q-1}}\right)\nonumber\\
  &+\frac{q-2}{8}\left\|
  |v|^{\frac{q}{2}-1}\nabla_H|v|\right\|_2^2,\label{I220}
\end{align}
with the constant $C$ independent of $q\geq 4$.
Recalling the elliptic estimate (\ref{p2}), the above inequality gives
\begin{align}
  I_{22}
  \leq&\frac{q-2}{8}\left\||v|^{\frac{q}{2}-1}\nabla_H|v|\right\|_2^2+C\left(1+\|v\|_2^2
  \right)^2\left(q\|v\|_q^{q-2}
  +\|v\|_q^{\frac{q(q-2)}{q-1}}\right).\label{I22}
\end{align}
As for $I_{21}$, recalling the elliptic estimate (\ref{p11}), and using the H\"older, Lemma \ref{lem2.4}, Ladyzhenskaya and Young inequalities, we deduce
\begin{align*}
  I_{21}=&-\int_\Omega
  p_s^1(x,y,t)\nabla_H\cdot(|v|^{q-2}v)dxdydz\\
  \leq&C_q\int_M|p_s^1(x,y,t)|\left(\int_{-h}^h|\nabla_Hv||v|^{q-2}dz\right)dxdy\\
  =&C_q\int_M|p_s^1(x,y,t)|\left(\int_{-h}^h|\nabla_Hv||v|^{\frac{q}{2}-1}|v|^{\frac{q}{2}-1}dz\right)dxdy\\
  \leq&C_q\int_M|p_s^1(x,y,t)|\left(\int_{-h}^h|\nabla_Hv|^2|v|^{q-2}dz\right)^{\frac{1}{2}}
  \left(\int_{-h}^h|v|^{q-2}dz\right)^{\frac{1}{2}}dxdy\\
  \leq&C_q\left(\int_M|p_s^1(x,y,t)|^{\frac{4q}{q+2}}dxdy\right)^{\frac{q+2}{4q}}\left(\int_M
  \int_{-h}^h|\nabla_Hv|^2|v|^{q-2}dzdxdy\right)^{\frac{1}{2}}\\
  &\times\bigg[\int_M\left(\int_{-h}^h
  |v|^{q-2}dz\right)^{\frac{2q}{q-2}}dxdy\bigg]^{\frac{q-2}{4q}}\\
  \leq&C_q\|p_s^1\|_{\frac{4q}{q+2},M}\left\||v|^{\frac{q}{2}-1}\nabla_Hv\right\|_2\bigg[\int_{-h}^h\left(\int_M
  |v|^{2q}dxdy\right)^{\frac{q-2}{2q}} dz\bigg]^{\frac{1}{2}}\\
  =&C_q\|p_s^1\|_{\frac{4q}{q+2},M}\left\||v|^{\frac{q}{2}-1}\nabla_Hv\right\|_2\left(\int_{-h}^h\|v\|_{2q,M}
  ^{q-2}dz\right)^{\frac{1}{2}}\\
  \leq&C_q\left(\int_{-h}^h\|v\|_{\frac{8q}{q+2},M}^2dz\right)\left\||v|^{\frac{q}{2}-1}\nabla_Hv\right\|_2
  \left(\int_{-h}^h\left\||v|^{\frac{q}{2}}\right\|_{4,M}^{\frac{2(q-2)}{q}}dz\right)^{\frac{1}{2}}\\
  \leq&C_q\left(\int_{-h}^h\|v\|_{4,M}\|v\|_{2q,M}dz\right)\left\||v|^{\frac{q}{2}-1}\nabla_Hv\right\|_2
  \left(\int_{-h}^h\left\||v|^{\frac{q}{2}}\right\|_{4,M}^{\frac{2(q-2)}{q}}dz\right)^{\frac{1}{2}}\\
  =&C_q\left(\int_{-h}^h\|v\|_{4,M}\left\||v|^{\frac{q}{2}}\right\|_{4,M}^{\frac{2}{q}}
  dz\right)\left\||v|^{\frac{q}{2}-1}\nabla_Hv\right\|_2
  \left(\int_{-h}^h\left\||v|^{\frac{q}{2}}\right\|_{4,M}^{\frac{2(q-2)}{q}}dz\right)^{\frac{1}{2}}\\
  \leq&C_q\left(\int_{-h}^h\|v\|_{2,M}^{\frac{1}{2}}\|v\|_{H^1(M)}^{\frac{1}{2}}\left\||v|
  ^{\frac{q}{2}}
  \right\|_{2,M}^{\frac{1}{q}}\left\||v|^{\frac{q}{2}}\right\|_{H^1(M)}^{\frac{1}{q}}dz\right)\\
  &\times\left\||v|^{\frac{q}{2}-1}\nabla_Hv\right\|_2
  \left(\int_{-h}^h\left\||v|^{\frac{q}{2}}\right\|_{2,M}^{\frac{q-2}{q}}
  \left\||v|^{\frac{q}{2}}\right\|_{H^1(M)}^{\frac{q-2}{q}}dz\right)^{\frac{1}{2}}\\
  \leq&C_q\|v\|_{2}^{\frac{1}{2}}\left(\|v\|_{2}^{\frac{1}{2}}+
  \|\nabla_Hv\|_{2}^{\frac{1}{2}}\right)\left\||v|^{\frac{q}{2}-1}\nabla_Hv\right\|_2\\
  &\times\left\||v|^{\frac{q}{2}}
  \right\|_{2}^{\frac{1}{2}}\left(\left\||v|^{\frac{q}{2}}
  \right\|_{2}^{\frac{1}{2}}+\left\|\nabla_H|v|^{\frac{q}{2}}
  \right\|_{2}^{\frac{1}{2}}\right)\\
  \leq&\frac{1}{8}\left\||v|^{\frac{q}{2}-1}\nabla_Hv\right\|_{2}^{2}+C_q[\|v\|_2^2(\|v\|_2^2+
  \|\nabla_Hv\|_2^2)+1]\|v\|_q^q\\
  \leq&\frac{1}{8}\left\||v|^{\frac{q}{2}-1}\nabla_Hv\right\|_{2}^{2}+C_q(1+\|v\|_2^2)(1+\|v\|_2^2+
  \|\nabla_Hv\|_2^2)\|v\|_q^q.
\end{align*}

By the aid of the above estimate, as well as (\ref{I1}), (\ref{I20}) and (\ref{I22}), it follows from the Young inequality that
\begin{align*}
  &I_1+I_2=I_1+I_{20}+I_{21}+I_{22}\\
  \leq&\frac{3}{8}\left((q-2)\left\||v|^{\frac{q}{2}-1}\nabla_H|v|\right\|_2^2
  +\left\||v|^{\frac{q}{2}-1}\nabla_Hv\right\|_{2}^{2}\right)\\
  &+Cq(1+\|T_0\|_q^2)\|v\|_q^{q-2}+Cq(1+\|v\|_2^2)^2
  (1+\|v\|_q^q)\\
  &+C_q(1+\|v\|_2^2)(1+\|v\|_2^2+\|\nabla_Hv\|_2^2)\|v\|_q^q\\
  \leq&\frac{3}{8}\left((q-2)\left\||v|^{\frac{q}{2}-1}\nabla_H|v|\right\|_2^2+\left\|
  |v|^{\frac{q}{2}-1}\nabla_Hv\right\|_{2}^{2}\right)\\
  &+C_q(1+\|v\|_2^2)(1+\|v\|_2^2+\|\nabla_Hv\|_2^2+\|T_0\|_q^2)(1+\|v\|_q^q).
\end{align*}
Substituting this into (\ref{roughqv1}), one obtains
\begin{align*}
  &\frac{d}{dt}\|v\|_q^q+\frac{5q}{8}\left(\left\|
  |v|^{\frac{q}{2}-1}\nabla_Hv\right\|_{2}^{2}+\varepsilon
  \left\|
  |v|^{\frac{q}{2}-1}\partial_zv\right\|_{2}^{2}\right)\\
  \leq& C_q(1+\|v\|_2^2)(1+\|v\|_2^2+\|\nabla_Hv\|_2^2+\|T_0\|_q^2)(1+\|v\|_q^q),
\end{align*}
from which, using (i), one obtains
\begin{align*}
  &\sup_{0\leq t\leq \mathcal T}\|v\|_q^q+\int_0^\mathcal{T}\left(\left\|
  |v|^{\frac{q}{2}-1}\nabla_Hv\right\|_{2}^{2}+\varepsilon
  \left\|
  |v|^{\frac{q}{2}-1}\partial_zv\right\|_{2}^{2}\right)dz\\
  \leq&e^{C_q\int_0^\mathcal T(1+\|v\|_2^2)(1+\|v\|_2^2+\|\nabla_Hv\|_2^2+\|T_0\|_q^2)ds}(1+\|v_0\|_q^q)\\
  \leq& e^{C_qe^{2\mathcal T}(\mathcal T+1)(1+\|v_0\|_2^2+\|T_0\|_2^2+\|T_0\|_q^2)^2}(1+\|v_0\|_q^q),
\end{align*}
in particular
\begin{equation}
  \sup_{0\leq t\leq \mathcal T}\|v\|_4^4+\int_0^\mathcal{T}\left\|
  |v|\nabla_Hv\right\|_{2}^{2}dz\leq K_1'(\mathcal T), \label{roughqv}
\end{equation}
where
$$
K_1'(\mathcal T)=e^{Ce^{2\mathcal T}(\mathcal T+1)(1+\|v_0\|_2^2+\|T_0\|_2^2+\|T_0\|_4^2)^2}(1+\|v_0\|_4^4),
$$
for a positive constant $C$ depending only on $h$.

\textbf{Step 2: the refined $L^\infty(0,\mathcal T; L^q(\Omega))$ estimate on $v$.} Noticing that all the constants $C$ in the estimates for $I_1, I_{20}$ and $I_{22}$ are independent of $q\in[4,\infty)$, it suffices to give a refined estimate for $I_{21}$.
Recalling the elliptic estimate (\ref{p12}), the similar argument as (\ref{I220}) yields
\begin{align*}
  I_{21}\leq&\int_M|\nabla_Hp_s^1(x,y,t)|\left(\int_{-h}^h|v|^{q-1}dz\right)dxdy\\
  \leq&\frac{q-2}{8}\left\||v|^{\frac{q}{2}-1}\nabla_H|v|\right\|_2^2+C(1+\|v\|_2^2)(1+\|\nabla_Hp_s^1
  \|_{2,M}^2)\left(q\|v\|_q^{q-2}+\|v\|_q^{\frac{q(q-2)}{q-1}}\right)\\
  \leq&\frac{q-2}{8}\left\||v|^{\frac{q}{2}-1}\nabla_H|v|\right\|_2^2+C(1+\|v\|_2^2)(1+\||v||\nabla_Hv|
  \|_{2}^2)\left(q\|v\|_q^{q-2}+\|v\|_q^{\frac{q(q-2)}{q-1}}\right),
\end{align*}
where the constant $C$ is independent of $q\in[4,\infty)$. Combining this with (\ref{I1}), (\ref{I20}) and (\ref{I22}), one obtains
\begin{align*}
  &I_1+I_2=I_1+I_{20}+I_{21}+I_{22}\\
  \leq&\frac{q-2}{2}\left\||v|^{\frac{q}{2}-1}\nabla_H|v|\right\|_2^2
  +\frac{1}{4}\left\||v|^{\frac{q}{2}-1}\nabla_H v\right\|_2^2\\
  &+Cq[1+\|T_0\|_q^2+(1+\|v\|_2^2)(1+\|v\|_2^2+\||v||\nabla_Hv|\|_{2}^2)]\|v\|_q^{q-2}\\
  &+C(1+\|v\|_2^2)(1
  +\|v\|_2^2+\||v||\nabla_Hv|\|_{2}^2)\|v\|_q^{\frac{q(q-2)}{q-1}}
  \\
  \leq&C[\|T_0\|_q^2+(1+\|v\|_2^2)(1+\|v\|_2^2+\||v||\nabla_Hv|\|_{2}^2)]\left(q\|v\|_q^{q-2}
  +\|v\|_q^{\frac{q(q-2)}{q-1}}\right)\\
  &+\frac{q-2}{2}\left\||v|^{\frac{q}{2}-1}\nabla_H|v|\right\|_2^2
  +\frac{1}{4}\left\||v|^{\frac{q}{2}-1}\nabla_H v\right\|_2^2,
\end{align*}
which, substituted into (\ref{roughqv1}), and using the Young inequality, gives
\begin{align*}
  \frac{d}{dt}(q+1+\|v\|_q^2)\leq&C[\|T_0\|_q^2+(1+\|v\|_2^2)(1+\|v\|_2^2+\||v||\nabla_Hv|\|_{2}^2)]
  \left(q+\|v\|_q^{\frac{q-2}{q-1}}\right)\\
  \leq&C[\|T_0\|_q^2+(1+\|v\|_2^2)(1+\|v\|_2^2+\||v||\nabla_Hv|\|_{2}^2)]
  \left(q+1+\|v\|_q^2\right).
\end{align*}
Recalling the estimate in (i) and (\ref{roughqv}), and applying the Gronwall inequality, it follows from the above inequality that
\begin{align*}
  (q+1)+\sup_{0\leq t\leq\mathcal T}\|v\|_q^2\leq& e^{C\int_0^\mathcal{T}[\|T_0\|_q^2+(1+\|v\|_2^2)(1+\|v\|_2^2+\||v||\nabla_Hv|\|_{2}^2)]dt}
  (q+1+\|v_0\|_q^2)\\
  \leq&2e^{C\int_0^\mathcal{T}[\|T_0\|_q^2+(1+\|v\|_2^2)(1+\|v\|_2^2+\||v||\nabla_Hv|\|_{2}^2)]dt}
  (1+\|v_0\|_q^2)q\\
  \leq&2e^{C[\|T_0\|_q^2\mathcal T+(1+K_0(\mathcal T))^2\mathcal T+(1+K_0(\mathcal T))K_1'(\mathcal T)]}(1+\|v_0\|_q)^2 q\\
  =:&K_1''(\mathcal T)e^{C\|T_0\|_q^2\mathcal T}(1+\|v_0\|_q)^2 q,
\end{align*}
where
$$
K_1''(\mathcal T)=2e^{C[(1+K_0(\mathcal T))^2\mathcal T+(1+K_0(\mathcal T))K_1'(\mathcal T)]},
$$
for a constant $C$ depends only on $h$, where $K_0(\mathcal T)$ and $K_1'(\mathcal T)$ are given as before.
Therefore, one obtains
$$
\sup_{0\leq t\leq\mathcal T}\|v\|_q \leq K_1(\mathcal T)e^{C\|T_0\|_q^2\mathcal T}
(1+\|v_0\|_q)\sqrt q,
$$
where $K_1(\mathcal T)$ is given by
\begin{equation}
  K_1(\mathcal T)=\sqrt{K_1''(\mathcal T)}, \label{K1}
\end{equation}
with $K_1''(\mathcal T)$ being given as before, proving (iii). This completes the proof.
\end{proof}

\section{High order energy inequalities}
\label{sec4}

In this section, we establish the energy inequalities for the derivatives up to second order of the strong solutions to system (\ref{eq1})--(\ref{eq3}), subject to the boundary and initial conditions (\ref{BC1})--(\ref{IC}). As we stated in the introduction and will also see below, the treatment of the different derivatives varies: we always work on the vertical derivatives first and then on the horizontal ones.

We first deal with energy inequalities for the first order derivatives which are described by the following proposition.

\begin{proposition}
  \label{prop4.1}
Let $(v,T)$ be the global strong solution of system (\ref{eq1})--(\ref{eq3}), subject to the boundary and initial conditions (\ref{BC1})--(\ref{IC}). Then for any $\mathcal T\in(0,\infty)$, we have the following:

(i) $L^\infty(0,\mathcal T; L^q(\Omega))$, $q\in[2,\infty)$, estimate of $u:=\partial_zv$:
\begin{equation*}
  \frac{d}{dt}\|u\|_q^q+\int_\Omega|u|^{q-2}\Big(|\nabla_Hu|^2+\varepsilon|\partial_zu|^2\Big)
  dxdydz\leq C_q(\|v\|_\infty^2+1)\left(\|u\|_q^q+1\right);
\end{equation*}

(ii) $L^\infty(0,\mathcal T; L^2(\Omega))$ estimate on $\partial_zT$:
\begin{align*}
  &\frac{d}{dt}\|\partial_zT\|_2^2+\|(\nabla_H\partial_zT,\sqrt\varepsilon\partial_z^2T)\|_2^2\\
  \leq& C\left(1+\|v\|_\infty^2\right)\|\partial_zT\|_2^2+C\|(\nabla_H v,u,\nabla_Hu)\|_2^2;
\end{align*}

(iii) $L^\infty(0,\mathcal T; L^2(\Omega))$ estimate on $\nabla_Hv$:
\begin{align*}
  &\frac{d}{dt}\|\nabla_Hv\|_2^2+\|(\Delta_Hv,\sqrt\varepsilon\nabla_H\partial_zv)\|_2^2\\
  \leq& C\|v\|_\infty^2\|\nabla_Hv\|_2^2+C\left(\|\nabla_HT\|_2^2+\|u\|_4^8\right);
\end{align*}

(iv) $L^\infty(0,\mathcal T; L^2(\Omega))$ estimate on $\nabla_HT$:
\begin{align*}
  &\frac{d}{dt}\|\nabla_HT\|_2^2+\|(\Delta_HT,\sqrt\varepsilon\nabla_H\partial_zT)\|_2^2\\
  \leq& C\left(1+\|(\partial_zT,\nabla_Hv)\|_2^2\right)^2
  \left(1+\|(\nabla_H\partial_zT,\Delta_Hv)\|_2^2\right);
\end{align*}
where $C$ is a positive constant depending only one $h, \mathcal T$ and the initial data (the constant $C_q$ in (i) depends also on $q$).
\end{proposition}

\begin{proof}
(i) Differentiating equation (\ref{eq1}) with respect to $z$, one can easily check that $u:=\partial_zv$ satisfies
\begin{equation}\label{u}
  \begin{split}
    &\partial_tu+(v\cdot\nabla_H)u-\left(\int_{-h}^z\nabla_H\cdot v(x,y,\xi,t)d\xi\right)\partial_zu-\Delta_Hu\\
    & -\varepsilon\partial_z^2u+(u\cdot\nabla_H)v-(\nabla_H\cdot v)u+f_0k\times u-\nabla_HT=0.
  \end{split}
\end{equation}
Multiplying the above equation by $|u|^{q-2}u$, and integrating over $\Omega$, it follows from integration by parts that
\begin{align*}
  &\frac{1}{q}\frac{d}{dt}\int_\Omega|u|^qdxdydz+\int_\Omega|u|^{q-2}\Big(|\nabla_Hu|^2
  +\varepsilon|\partial_zu|^2\\
  &+(q-2)|\nabla_H|u||^2
  +(q-2)\varepsilon|\partial_z|u||^2\Big)dxdydz\\
  =&\int_\Omega((\nabla_H\cdot v)u-(u\cdot\nabla_H)v+\nabla_HT)\cdot|u|^{q-2}udxdydz\\
  \leq&C_q\int_\Omega(|v||u|^{q-1}|\nabla_Hu|+|T||u|^{q-2}|\nabla_Hu|)dxdydz\\
  \leq&\frac{1}{2}\int_\Omega|u|^{q-2}|\nabla u|^2dxdydz+C_q\int_\Omega(|v|^2|u|^q+|T|^2|u|^{q-2})dxdydz.
\end{align*}
Recalling that $\sup_{0\leq t\leq\mathcal T}\|T\|_\infty\leq C$, guaranteed by Proposition \ref{prop3.1} (ii), we have
\begin{align*}
  &\frac{d}{dt}\|u\|_q^q+\int_\Omega|u|^{q-2}\Big(|\nabla_Hu|^2+\varepsilon|\partial_zu|^2\Big)dxdydz\\
  \leq& C_q(\|v\|_\infty^2+\|T\|_\infty^2)(\|u\|_q^q+1)\leq C_q(\|v\|_\infty^2+1)(\|u\|_q^q+1),
\end{align*}
proving (i).

(ii) Multiplying equation (\ref{eq3}) by $-\partial_z^2T$, and integrating over $\Omega$, then it follows from integrating by parts, and using $\sup_{0\leq t\leq \mathcal T}\|T\|_\infty\leq C$ guaranteed by Proposition \ref{prop3.1} (ii) that
\begin{align*}
  &\frac{1}{2}\frac{d}{dt}\int_\Omega|\partial_zT|^2dxdydz+\int_\Omega
  \Big(|\nabla_H\partial_zT|^2+\varepsilon|\partial_z^2T|^2\Big)dxdydz\\
  =&\int_\Omega\left[(v\cdot\nabla_H)T-\left(\int_{-h}^z\nabla_H\cdot vd\xi\right)\left(\partial_zT+\frac{1}{h}\right)\right]\partial_z^2Tdxdydz\\
  =&-\int_\Omega\left[u\cdot\nabla_HT-(\nabla_H\cdot v)\partial_zT-h^{-1}(\nabla_H\cdot v)\right]\partial_zTdxdydz\\
  =&\int_\Omega[\nabla_H\cdot(u\partial_zT)T-v\cdot\nabla_H|\partial_zT|^2+h^{-1}(\nabla_H\cdot v)\partial_zT]dxdydz\\
  \leq&C\int_\Omega(|\nabla_Hu||\partial_zT||T|+|u||\nabla_H\partial_zT||T|+|v||\partial_zT||\nabla_H
  \partial_zT|+|\nabla_Hv||\partial_zT|)dxdydz\\
  \leq&\frac{1}{2}\|\nabla_H\partial_zT\|_2^2+C\|(\nabla_Hu,u,\partial_zT,
  \nabla_Hv)\|_2^2+C\|v\|_\infty^2\|\partial_zT\|_2^2\\
  \leq&\frac{1}{2}\|\nabla_H\partial_zT\|_2^2+C\|(\nabla_Hu,u,\nabla_Hv)\|_2^2
  +C(1+\|v\|_\infty^2)\|\partial_zT\|_2^2,
\end{align*}
from which one obtains (ii).

(iii) Multiplying equation (\ref{eq1}) by $-\Delta_Hv$, and integrating over $\Omega$, then it follows from integrating by parts and the Cauchy inequality that
\begin{align}
  &\frac{1}{2}\frac{d}{dt}\int_\Omega|\nabla_Hv|^2dxdydz+\int_\Omega\Big(
  |\Delta_Hv|^2+\varepsilon|\nabla_H\partial_zv|^2\Big)dxdydz\nonumber\\
  =&\int_\Omega\left[(v\cdot\nabla_H)v-\left(\int_{-h}^z\nabla_H\cdot vd\xi\right)u-\nabla_H\left(\int_{-h}^zTd\xi\right)\right]\cdot\Delta_Hvdxdydz\nonumber\\
  \leq&C(\|v\|_\infty\|\nabla_Hv\|_2\|\Delta_Hv\|_2+\|\nabla_HT\|_2\|\Delta_Hv\|_2)\nonumber\\
  &+C\int_M\left(\int_{-h}^h|\nabla_Hv|dz\right)\left(\int_{-h}^h|u||\Delta_Hv|dz\right)dxdy\nonumber\\
  \leq&\frac{1}{4}\|\Delta_Hv\|_2^2+C\|v\|_\infty^2\|\nabla_Hv\|_2^2+C\|\nabla_HT\|_2^2\nonumber\\
  &+C\int_M\left(\int_{-h}^h|\nabla_Hv|dz\right)\left(\int_{-h}^h|u||\Delta_Hv|dz\right)dxdy.
  \label{esthv1}
\end{align}
It follows from the H\"older, Lemma \ref{lem2.4}, Gagliardo-Nirenberg and Young inequalities that
\begin{align}
  &C\int_M\left(\int_{-h}^h|\nabla_Hv|dz\right)
  \left(\int_{-h}^h|u||\Delta_Hv|dz\right)dxdy\nonumber\\
  \leq&C\int_M\left(\int_{-h}^h|\nabla_Hv|dz\right)\left(\int_{-h}^h|u|^2dz\right)^{\frac{1}{2}}
  \left(\int_{-h}^h|\Delta_Hv|^2dz\right)^{\frac{1}{2}}dxdy\nonumber\\
  \leq&C\left[\int_M\left(\int_{-h}^h|\nabla_Hv|dz\right)^4dxdy\right]^{\frac{1}{4}}
  \left[\int_M\left(\int_{-h}^h|u|^2dz\right)^2dxdy\right]^{\frac{1}{4}}\nonumber\\
  &\times\left(\int_\Omega|\Delta_Hv|^2dxdydz\right)^{\frac{1}{2}}\nonumber\\
  \leq&C\left(\int_{-h}^h\|\nabla_Hv\|_{4,M}dz\right)\left[\int_{-h}^h\left(\int_M|u|^4dz\right)^{\frac{1}{2}}
  dxdy\right]^{\frac{1}{2}}\|\Delta_Hv\|_2\nonumber\\
  =&C\left(\int_{-h}^h\|\nabla_Hv\|_{4,M}dz\right)\left(\int_{-h}^h\|u\|_{4,M}^2dz\right)^{\frac{1}{2}}
  \|\Delta_Hv\|_2\nonumber\\
  \leq&C\left(\int_{-h}^h\|v\|_{2,M}^{\frac{1}{4}}\|\nabla_H^2v\|_{2,M}^{\frac{3}{4}}dz\right)\|u\|_4
  \|\Delta_Hv\|_2\nonumber\\
  \leq&C\|v\|_2^{\frac{1}{4}}\|u\|_4\|\Delta_Hv\|_2^{\frac{7}{4}}\leq\frac{1}{4}\|\Delta_Hv\|_2^2
  +C\|v\|_2^2\|u\|_4^8.\label{4.2-1}
\end{align}
Substituting this into (\ref{esthv1}), and recalling that $\sup_{0\leq t\leq\mathcal T}\|v\|_2^2\leq C$ guaranteed by Proposition \ref{prop3.1} (i), one obtains
\begin{equation*}
  \frac{d}{dt}\|\nabla_Hv\|_2^2+\|(\Delta_Hv,\sqrt\varepsilon\nabla_H\partial_zv)\|_2^2\leq C\|v\|_\infty^2\|\nabla_Hv\|_2^2+C(\|\nabla_HT\|_2^2+\|u\|_4^8),
\end{equation*}
proving (iii).

(iv) Multiplying equation (\ref{eq3}) by $-\Delta_H T$, and integrating over $\Omega$, then it follows from integrating by parts and the Cauchy inequality that
\begin{align}
  &\frac{1}{2}\frac{d}{dt}\int_\Omega|\nabla_HT|^2dxdydz+\int_\Omega\Big(|\Delta_H T|^2+\varepsilon|\nabla_H\partial_zT|^2\Big)dxdydz\nonumber\\
  =&\int_\Omega\left[v\cdot\nabla_HT-\left(\int_{-h}^z\nabla_H\cdot vd\xi\right)\left(\partial_zT+\frac{1}{h}\right)\right]\Delta_H Tdxdydz\nonumber\\
  \leq&C(\|v\|_\infty\|\nabla_HT\|_2\|\Delta_HT\|_2+\|\nabla_Hv\|_2\|\Delta_HT\|_2)\nonumber\\
  &-\int_\Omega\left(\int_{-h}^z\nabla_H\cdot vd\xi\right)\partial_zT\Delta_HTdxdydz\nonumber\\
  \leq&\frac{1}{4}\|\Delta_HT\|_2^2+C\|v\|_\infty^2\|\nabla_HT\|_2^2+C\|\nabla_Hv\|_2^2\nonumber\\
  &-\int_\Omega\left(\int_{-h}^z\nabla_H\cdot vd\xi\right)\partial_zT\Delta_HTdxdydz. \label{estht1}
\end{align}
It follows from integrating by parts, Lemma \ref{lad} and Proposition \ref{prop3.1} (ii) that
\begin{align*}
  &-\int_\Omega\left(\int_{-h}^z\nabla_H\cdot vd\xi\right)\partial_zT\Delta_HTdxdydz\\
  =&\int_\Omega\left(\int_{-h}^z\nabla_H\nabla_H\cdot vd\xi\partial_zT+\int_{-h}^z\nabla_H\cdot vd\xi\nabla_H\partial_zT\right)\cdot\nabla_HTdxdydz\\
  \leq&C\int_M\left(\int_{-h}^h|\nabla_H^2v|dz\right)\left(\int_{-h}^h|\partial_zT||\nabla_HT|dz
  \right)dxdy\\
  &+C\int_M\left(\int_{-h}^h|\nabla_Hv|dz\right)\left(\int_{-h}^h|\nabla_H\partial_zT|
  |\nabla_HT|dz\right)dxdy\\
  \leq&C\|\nabla_H^2v\|_2\|\partial_zT\|_2^{\frac{1}{2}}\left(\|\partial_zT\|_2^{\frac{1}{2}}+\|\nabla_H
  \partial_zT\|_2^{\frac{1}{2}}\right)\|T\|_\infty^{\frac{1}{2}}\|\nabla_H^2T\|_2^{\frac{1}{2}}\\
  &+C\|\nabla_Hv\|_2^{\frac{1}{2}}\|\nabla_H^2v\|_2^{\frac{1}{2}}\|\nabla_H\partial_zT\|_2
  \|T\|_\infty^{\frac{1}{2}}\|\nabla_H^2T\|_2^{\frac{1}{2}}\\
  \leq&\frac{1}{4}\|\nabla_H^2T\|_2^2+C[\|\nabla_H^2v\|_2^2+\|\partial_zT\|_2^2(\|\partial_z
  T\|_2^2+\|\nabla_H\partial_zT\|_2^2)\\
  &+\|\nabla_H\partial_zT\|_2^2+\|\nabla_Hv\|_2^2\|\nabla_H^2v\|_2^2]\\
  \leq&\frac{1}{4}\|\Delta_HT\|_2^2+C\left(1+\|(\partial_zT,\nabla_Hv)\|_2^2\right)^2\left( 1+\|(\nabla_H\partial_zT,\Delta_Hv)\|_2^2\right).
  \end{align*}
Substituting this into (\ref{estht1}) yields
\begin{align*}
  &\frac{d}{dt}\|\nabla_HT\|_2^2+\|(\Delta_HT,\sqrt\varepsilon\nabla_H\partial_T)\|_2^2\\
  \leq& C\Big(1+\|(\partial_zT,\nabla_Hv)\|_2^2\Big)^2
  \Big(1+\|(\nabla_H\partial_zT,\Delta_Hv)\|_2^2\Big),
\end{align*}
proving (iv).
\end{proof}

Now we consider the energy inequalities for the second order derivatives. We have the following proposition.

\begin{proposition}
  \label{prop4.2}
Let $(v,T)$ be the global strong solution of system (\ref{eq1})--(\ref{eq3}) subject to the boundary and initial conditions (\ref{BC1})--(\ref{IC}). Then for any $\mathcal T\in(0,\infty)$, we have the following:

(i) $L^\infty(0,\mathcal T; L^2(\Omega))$ estimate on $\partial_zu$:
\begin{align*}
  &\frac{d}{dt}\|\partial_zu\|_2^2+\|(\nabla_H\partial_zu,\sqrt\varepsilon
  \partial_z^2u)\|_2^2\\
  \leq& C\Big(1+\|v\|_\infty^2\Big)\|\partial_zu\|_2^2+C\Big(\||u||\nabla_Hu|\|_2^2
  +\|\nabla_H\partial_zT\|_2^2\Big);
\end{align*}

(ii) $L^\infty(0,\mathcal T; L^2(\Omega))$ estimate on $\partial_z^2T$:
\begin{align*}
  &\frac{d}{dt}\|\partial_z^2T\|_2^2+\|(\nabla_H\partial_z^2T,\sqrt\varepsilon
  \partial_z^3T)\|_2^2\nonumber\\
  \leq&C\Big(1+\|v\|_\infty^2\Big)\|\partial_z^2T\|_2^2+C\Big(1+\|(\partial_z u,\partial_z T)\|_2^2\Big)^2\nonumber\\
  &\times\Big(1+\|(\nabla_Hu,\nabla_H\partial_zu,
  \nabla_H\partial_zT)\|_2^2\Big);
\end{align*}

(iii) $L^\infty(0,\mathcal T; L^2(\Omega))$ estimate on $\nabla_Hu$:
\begin{align*}
  &\frac{d}{dt}\|\nabla_Hu\|_2^2+\|(\Delta_Hu,\sqrt\varepsilon\nabla_H\partial_zu)\|_2^2\\
  \leq&C\|v\|_\infty^2\|\nabla_Hu\|_2^2+C\|(u,\nabla_Hv,\partial_zu,\nabla_HT)\|_2^2\\
  &\times
  \|(\nabla_Hu,\Delta_Hv,\nabla_H\partial_zu)\|_2^2;
\end{align*}

(iv) $L^\infty(0,\mathcal T; L^2(\Omega))$ estimate on $\nabla_H\partial_zT$:
\begin{align*}
  &\frac{d}{dt}\|\nabla_H\partial_zT\|_2^2+\|(\Delta_H\partial_zT,\nabla_H\partial_z^2T)
  \|_2^2\nonumber\\
  \leq&C\|v\|_\infty^2\|\nabla_H\partial_zT\|_2^2+C\Big(1+\|(\nabla_Hv,
  \nabla_HT,\partial_zu,\partial_z^2T)\|_2^2\Big)^2\\
  &\times\Big(1+\|(\Delta_Hv,\Delta_HT,
  \nabla_H\partial_zu,\nabla_H\partial_z^2
  T)\|_2^2\Big);
\end{align*}

(v) $L^\infty(0,\mathcal T; L^2(\Omega))$ estimate on $\Delta_H v$ and $\Delta_HT$:
\begin{align*}
  &\frac{d}{dt}\|(\Delta_H v,\Delta_HT)\|_2^2+\|(\nabla_H\Delta_H v,\nabla_H\Delta_HT,\sqrt\varepsilon\partial_z\Delta_Hv,\sqrt\varepsilon
  \partial_z\Delta_HT)\|_2^2\nonumber\\
  \leq&C\Big(1+\|v\|_\infty^2\Big)\|(\Delta_Hv,\Delta_HT)\|_2^2+C\Big(1+\|(\nabla_Hv,\nabla_H
  T,\nabla_Hu,\partial_z^2T,\nabla_H\partial_zT)\|_2^2\\
  &+\|u\|_4^4\Big)^3\Big(1+\|(\Delta_Hv,\Delta_HT,\Delta_Hu,\Delta_H\partial_zT)
  \|_2^2\Big),
\end{align*}
where $C$ is a positive constant depending only on $h,\mathcal T$ and the initial data.
\end{proposition}

\begin{proof}
(i) Differentiating equation (\ref{u}) with respect to $z$, one can easily check that $\partial_zu$ satisfies
\begin{eqnarray*}
  &\partial_t\partial_zu+(v\cdot\nabla_H)\partial_zu-\left(\int_{-h}^z\nabla_H\cdot v(x,y,\xi,t)d\xi\right)\partial_z^2u
\\
  &-\Delta_H\partial_zu-\varepsilon\partial_z^3u+2(u\cdot\nabla_H)u-(\nabla_H\cdot u)u+(\partial_zu\cdot\nabla_H)v\\
  &-2(\nabla_H\cdot v)\partial_zu+f_0k\times\partial_zu-\nabla_H\partial_zT=0.
\end{eqnarray*}
Multiplying the above equation by $\partial_zu$, and integrating over $\Omega$, then it follows from integrating by parts and the Cauchy inequality that
\begin{align*}
  &\frac{1}{2}\frac{d}{dt}\int_\Omega|\partial_zu|^2dxdydz
  +\int_\Omega\Big(|\nabla_H\partial_zu|^2+\varepsilon|\partial_z^2u|^2\Big)dxdydz\\
  =&\int_\Omega[(\nabla_H\cdot u)u-2(u\cdot\nabla_H)u+2(\nabla_H\cdot v)\partial_zu-(\partial_zu\cdot\nabla_H)v+\nabla_H\partial_zT]\cdot \partial_zudxdydz\\
  =&\int_\Omega \{[(\nabla_H\cdot u)u-2(u\cdot\nabla_H)u+\nabla_H\partial_zT]\cdot\partial_zu -2v\cdot\nabla_H|\partial_zu|^2\\
  &+\nabla_H\cdot(\partial_zu\otimes\partial_zu)\cdot v \}dxdydz\\
  \leq&C\int_\Omega(|u||\nabla_Hu||\partial_zu|+|v||\partial_zu||\nabla_H\partial_zu|+
  |\nabla_H\partial_zT||\partial_zu|)dxdydz\\
  \leq&C(\||u|\nabla_Hu\|_2\|\partial_zu\|_2+\|v\|_\infty\|\partial_zu\|_2\|\nabla_H\partial_zu\|_2
  +\|\nabla_H\partial_zT\|_2\|\partial_zu\|_2)\\
  \leq&\frac{1}{2}\|\nabla_H\partial_zu\|_2^2+C(\||u|\nabla_Hu\|_2^2+\|\nabla_H\partial_zT\|_2^2)
  +C(1+\|v\|_\infty^2)\|\partial_zu\|_2^2,
\end{align*}
from which, one obtains (i).

(ii) Differentiating equation (\ref{eq3}) with respect to $z$ yields
\begin{equation}
  \label{zt}
  \begin{split}
  &\partial_t\partial_zT+v\cdot\nabla_H\partial_zT-\left(\int_{-h}^z\nabla_H\cdot vd\xi\right)\partial_z^2T+u\cdot\nabla_HT\\
  &\quad \quad-(\nabla_H\cdot v)\left(\partial_zT+\frac{1}{h}\right)-\Delta_H\partial_zT-\varepsilon\partial_z^3T=0.
  \end{split}
\end{equation}
Multiplying the above equation by $-\partial_z^3T$, integrating over $\Omega$, and using the facts
\begin{eqnarray*}
  &|\nabla_H\partial_zT(x,y,z,t)|\leq \int_{-h}^h|\nabla_H\partial_z^2T(x,y,\xi,t)|d\xi,\\
  &|u(x,y,z,t)|\leq \int_{-h}^h|\partial_zu(x,y,\xi,t)|d\xi,
\end{eqnarray*}
it follows from integrating by parts, Lemma \ref{lad}, Proposition \ref{prop3.1} and the Young inequality that
\begin{align*}
  &\frac{1}{2}\frac{d}{dt}\int_\Omega|\partial_z^2T|^2dxdydz+
  \int_\Omega\Big(|\nabla_H\partial_z^2T|^2+\varepsilon|\partial_z^3T|^2\Big)
  dxdydz\\
  =&\int_\Omega\left[v\cdot\nabla_H\partial_zT- \int_{-h}^z\nabla_H\cdot vd\xi \partial_z^2T+u\cdot\nabla_HT-\nabla_H\cdot v(\partial_zT+h^{-1})\right]\partial_z^3Tdxdydz\\
  =&-\int_\Omega[2u\cdot\nabla_H\partial_zT-2\nabla_H\cdot v\partial_z^2T+\partial_zu\cdot\nabla_HT -\nabla_H\cdot u (\partial_zT+h^{-1} ) ]\partial_z^2Tdxdydz\\
  =&\int_\Omega[2\partial_z(u\cdot\nabla_H\partial_zT)\partial_zT-2v\cdot\nabla_H|\partial_z^2T|^2+ \nabla_H\cdot(\partial_zu\partial_z^2T)T\\
  &+(\nabla_H\cdot u)\frac{1}{2}\partial_z|\partial_zT|^2+h^{-1}(\nabla_H\cdot u)\partial_z^2T]dxdydz\\
  =&\int_\Omega[2(\partial_zu\cdot\nabla_H\partial_zT+u\cdot\nabla_H\partial_z^2T)\partial_zT-4v\cdot \nabla_H\partial_z^2T\partial_z^2T\\
  &+(\nabla_H\cdot\partial_zu\partial_z^2T+\partial_zu\cdot\nabla_H\partial_z^2T)T+\frac{1}{2}\partial_z u\cdot\nabla_H|\partial_zT|^2+h^{-1}\nabla_H\cdot u\partial_z^2T]dxdydz\\
  =&\int_\Omega[3\partial_zu\cdot\nabla_H\partial_zT\partial_zT+2u\cdot\nabla_H\partial_z^2T\partial_zT -4v\cdot\nabla_H\partial_z^2T\partial_z^T\\
  &+(\nabla_H\cdot\partial_zu\partial_z^2T+\partial_zu\cdot\nabla_H\partial_z^2T)T+h^{-1}\nabla_H\cdot u\partial_z^2T]dxdydz\\
  \leq&C\int_M\left(\int_{-h}^h
  |\nabla_H\partial_z^2T|dz\right)\left(\int_{-h}^h|\partial_zu||\partial_zT|dz\right)dxdy \\
  &+C\int_M\left(\int_{-h}^h|\partial_zu|dz\right)\left(\int_{-h}^h|\partial_zT||\nabla_H
  \partial_z^2T|dz\right)dxdy+C(\|v\|_\infty\|\partial_z^2T\|_2\|\nabla_H
  \partial_z^2T\|_2\\
  &+\|\nabla_H\partial_zu\|_2\|\partial_z^2T\|_2
  +\|\partial_zu\|_2\|\nabla_H\partial_z^2T\|_2+\|\nabla_Hu\|_2\|\partial_z^2T\|_2)\\
  \leq&C\|\nabla_H\partial_z^2T\|_2
  \|\partial_zu\|_2^{\frac{1}{2}}(\|\partial_zu\|_2^{\frac{1}{2}}+\|\nabla_H
  \partial_zu\|_2^{\frac{1}{2}})\|\partial_zT\|_2^{\frac{1}{2}}(\|\partial_zT\|_2^{\frac{1}{2}}+
  \|\nabla_H\partial_zT\|_2^{\frac{1}{2}}) \\
  &+C(\|v\|_\infty\|\partial_z^2T\|_2\|\nabla_H
  \partial_z^2T\|_2+\|\nabla_H\partial_zu\|_2\|\partial_z^2T\|_2\\
  &+\|\partial_zu\|_2\|\nabla_H\partial_z^2T\|_2+ \|\nabla_Hu\|_2\|\partial_z^2T\|_2)\\
  \leq&\frac{1}{2}\|\nabla_H\partial_z^2T\|_2^2+C[\|\partial_zu\|_2^2(\|\partial_zu\|_2^2+\|\nabla_H\partial_z
  u\|_2^2)+\|\partial_zT\|_2^2(\|\partial_zT\|_2^2+\|\nabla_H\partial_zT\|_2^2)]\\
  &+C(\|v\|_\infty^2+1)\|\partial_z^2T\|_2^2+
  C(\|\nabla_Hu\|_2^2+\|\nabla_H\partial_zu\|_2^2+\|\partial_zu\|_2^2)\\
  \leq&\frac{1}{2}\|\nabla_H\partial_z^2T\|_2^2+C(\|v\|_\infty^2+1)\|\partial_z^2T\|_2^2+C
  (1+\|\partial_zu\|_2^2+\|\partial_zT\|_2^2)^2\\
  &\times(1+\|\nabla_Hu\|_2^2+\|\nabla_H\partial_zu\|_2^2+\|\nabla_H\partial_zT\|_2^2),
\end{align*}
from which one obtains (ii).

(iii) Multiplying equation (\ref{u}) by $-\Delta_Hu$, integrating over $\Omega$, and using the fact that
$|u(x,y,z,t)|\leq\int_{-h}^h|\partial_zu(x,y,\xi,t)|d\xi$, it follows from
integrating by pars, Lemma \ref{lad} and the Young inequality that
\begin{align*}
  &\frac{1}{2}\frac{d}{dt}\int_\Omega|\nabla_Hu|^2dxdydz
  +\int_\Omega\Big(|\Delta_Hu|^2+\varepsilon|\nabla_H\partial_zu|^2\Big)dxdydz\\
  =&\int_\Omega\left[(v\cdot\nabla_H)u-\int_{-h}^z\nabla_H\cdot vd\xi\partial_zu+(u\cdot\nabla_H)v-(\nabla_H\cdot v)u-\nabla_HT\right]\cdot\Delta_Hudxdydz\\
  \leq&\int_\Omega\bigg[|v||\nabla_Hu||\Delta_Hu|+\left(\int_{-h}^h|\nabla_Hv|dz\right)
  |\partial_zu||\Delta_Hu|\\
  &+|u||\nabla_Hv||\Delta_Hu|+|\nabla_HT||\Delta_Hu|\bigg]dxdydz\\
  \leq
  &\|v\|_\infty\|\nabla_Hu\|_2\|\Delta_Hu\|_2+C\int_M\left(\int_{-h}^h|\nabla_Hv|dz\right)\left(\int_{-h}^h |\partial_zu||\Delta_Hu|dz\right)dxdy\\
  &+C\int_M\left(\int_{-h}^h|\partial_zu|dz\right)\left(\int_{-h}^h|\nabla_Hv||\Delta_Hu|dz\right) dxdy+\|\nabla_HT\|_2\|\Delta_Hu\|_2\\
  \leq&\|v\|_\infty\|\nabla_Hu\|_2\|\Delta_Hu\|_2
  +C\|\nabla_Hv\|_2^{\frac{1}{2}}\|\nabla^2_Hv\|_2^{\frac{1}{2}}\|
  \partial_zu\|_2^{\frac{1}{2}}(\|
  \partial_zu\|_2^{\frac{1}{2}}+\|\nabla_H\partial_zu\|_2^{\frac{1}{2}})\|\Delta_Hu\|_2\\
  &+C\|\partial_zu\|_2^{\frac{1}{2}}(\|\partial_zu\|_2^{\frac{1}{2}}+\|\nabla_H\partial_zu\|_2
  ^{\frac{1}{2}})
  \|\nabla_Hv\|_2^{\frac{1}{2}}\|\nabla_H^2v\|_2^{\frac{1}{2}}\|\Delta_Hu\|_2
  +\|\nabla_HT\|_2\|\Delta_Hu\|_2\\
  \leq&\frac{1}{2}\|\Delta_Hu\|_2^2+C\|v\|_\infty^2\|\nabla_Hu\|_2^2+C[\|\partial_zu\|_2^2(\|\partial_zu\|_2^2+\|\nabla_H
  \partial_zu\|_2^2)
\\
  &+\|\nabla_Hv\|_2^2\|\Delta_Hv\|_2^2+\|\nabla_HT\|_2^2]\\
  \leq&\frac{1}{2}\|\Delta_Hu\|_2^2+C\|v\|_\infty^2\|\nabla_Hu\|_2^2+C(1+\|\nabla_Hv\|_2^2
\\
  &+\|\nabla_HT\|_2^2+\|\partial_zu\|_2^2)^2(1+\|\Delta_Hv\|_2^2+\|\nabla_H\partial_zu\|_2^2),
\end{align*}
from which, one obtains (iii).

(iv) Multiplying equation (\ref{zt}) by $-\Delta_H\partial_zT$, integrating the resulting equation over $\Omega$, and using the facts that
\begin{eqnarray*}
  &|u(x,y,z,t)|\leq\int_{-h}^h|\partial_zu(x,y,\xi,t)|d\xi,\\
  &|\partial_zT(x,y,z,t)|\leq\int_{-h}^h|\partial_z^2T(x,y,\xi,t)|d\xi,
\end{eqnarray*}
it follows from integration by parts, Lemma \ref{lad} and the Young inequality that
\begin{align*}
  &\frac{1}{2}\frac{d}{dt}\int_\Omega|\nabla_H\partial_zT|^2dxdydz+
  \int_\Omega\Big(|\Delta_H\partial_z
  T|^2+\varepsilon|\nabla_H\partial_z^2T|^2\Big)dxdydz\\
  =&\int_\Omega\bigg[v\cdot\nabla_H\partial_zT- \int_{-h}^z\nabla_H\cdot vd\xi \partial_z^2T+u\cdot\nabla_HT -\nabla_H\cdot v(\partial_zT+h^{-1})\bigg]\Delta_H\partial_zTdxdydz\\
  \leq& \int_\Omega\bigg[|v||\nabla_H\partial_zT|+ \bigg(\int_{-h}^h|\nabla_Hv|dz\bigg)
  |\partial_z^2T|+|u||\nabla_HT| \\
  &+|\nabla_Hv||\partial_zT|+|\nabla_Hv|\bigg]|\Delta_H\partial_zT|dxdydz
  \\
  \leq& \|v\|_\infty\|\nabla_H\partial_zT\|_2\|\Delta_H\partial_zT\|_2+ \int_\Omega\left[ \left(\int_{-h}^h|\nabla_Hv|dz\right) |\partial_z^2T|+
  \left(\int_{-h}^h|\partial_zu|dz\right) |\nabla_HT|\right.\\
  &\left.+|\nabla_Hv|\left(\int_{-h}^h|\partial_z^2T|dz\right)\right]|\Delta_H\partial_zT|dxdydz
  + \|\nabla_Hv\|_2\|\Delta_H
  \partial_zT\|_2 \\
  \leq& \|v\|_\infty\|\nabla_H\partial_zT\|_2\|\Delta_H\partial_zT\|_2+ \|\nabla_Hv\|_2\|\Delta_H
  \partial_zT\|_2 \\
  &+\int_M\left(\int_{-h}^h|\nabla_Hv|dz\right)\left(\int_{-h}^h|\partial_z^2T||\Delta_H\partial_zT|dz
  \right)dxdy\\
  &+\int_M\left(\int_{-h}^h|\partial_zu|dz\right)\left(\int_{-h}^h|\nabla_HT||\Delta_H\partial_zT|
  dz\right)dxdy\\
  &+\int_M\left(\int_{-h}^h|\partial_z^2T|dz\right)\left(\int_{-h}^h|\nabla_Hv||\Delta_H\partial_zT| dz\right)dxdy\\
  \leq&C[\|v\|_\infty\|\nabla_H\partial_zT\|_2\|\Delta_H\partial_zT\|_2
  +\|\nabla_Hv\|_2\|\Delta_H\partial_zT\|_2\\
  &+\|\nabla_Hv\|_2^{\frac{1}{2}}\|\nabla_H^2v\|_2^{\frac{1}{2}}\|\partial_z^2T\|_2^{\frac{1}{2}}
  (\|\partial_z^2T\|_2^{\frac{1}{2}}+\|\nabla_H\partial_z^2T\|_2^{\frac{1}{2}})\|\Delta_H\partial_zT\|_2\\
  &+\|\partial_zu\|_2^{\frac{1}{2}}(\|\partial_zu\|_2^{\frac{1}{2}}+\|\nabla_H\partial_zu\|_2^{\frac{1}{2}})
  \|\nabla_HT\|_2^{\frac{1}{2}}\|\nabla_H^2T\|_2^{\frac{1}{2}}\|\Delta_H\partial_zT\|_2
  ]\\
  \leq&\frac{1}{2}\|\Delta_H\partial_zT\|_2^2+C[ \|v\|_\infty^2\|\nabla_H\partial_zT\|_2^2
  +\|\nabla_Hv\|_2^2+\|\nabla_Hv\|_2^2\|\Delta_Hv\|_2^2+\|\nabla_HT\|_2^2\|\Delta_HT\|_2^2\\
  &+\|\partial_z^2T\|_2^2(\|\partial_z^2T\|_2^2+\|\nabla_H
  \partial_z^2T\|_2^2) +\|\partial_zu\|_2^2(\|\partial_zu\|_2^2+
  \|\nabla_H\partial_zu\|_2^2)]\\
  \leq&\frac{1}{2}\|\Delta_H\partial_zT\|_2^2+C[\|v\|_\infty^2\|\nabla_H\partial_zT\|_2^2+
  (1+\|\nabla_Hv\|_2^2+\|\nabla_HT\|_2^2+\|\partial_z^2T\|_2^2+\|\partial_zu\|_2^2)^2\\
  &\times(1+\|\Delta_Hv\|_2^2+\|\Delta_HT\|_2^2
  +\|\nabla_H\partial_z^2T\|_2^2+\|\nabla_H\partial_zu\|_2^2)],
\end{align*}
from which one obtains (iv).

(v) Applying the operator $\nabla_H$ to equations (\ref{eq1}) and (\ref{eq3}), multiplying the resulting equations by $-\nabla_H\Delta_Hv$ and $-\nabla_H\Delta_HT$, respectively, and noticing
that
$$
|\nabla_Hv(x,y,z,t)|\leq\frac{1}{2h}\int_{-h}^h|\nabla_Hv(x,y,\xi,t)|d\xi+\int_{-h}^h|\nabla_H
u(x,y,\xi,t)|d\xi,
$$
it follows from integration by parts
\begin{align}
  &\frac{1}{2}\frac{d}{dt}\int_\Omega(|\Delta_Hv|^2+|\Delta_HT|^2)dxdydz\nonumber\\
  &+\int_\Omega(|\nabla_H\Delta
  _Hv|^2+|\nabla_H\Delta_HT|^2+\varepsilon|\partial_z\Delta_Hv|^2+\varepsilon|\partial_z
  \Delta_HT|^2)dxdydz\nonumber\\
  =&\int_\Omega\nabla_H\left((v\cdot\nabla_H)v-\left(\int_{-h}^z\nabla_H\cdot vd\xi\right)\partial_zv-\nabla_H\left(\int_{-h}^zTd\xi\right)\right):\nabla_H\Delta_Hvdxdydz\nonumber\\
  &+\int_\Omega\nabla_H\left(v\cdot\nabla_HT-\left(\int_{-h}^z\nabla_H\cdot vd\xi\right)
  \left(\partial_zT+\frac{1}{h}\right)\right)\cdot\nabla_H\Delta_HTdxdydz\nonumber\\
  \leq&\int_\Omega\bigg[|v|(|\nabla_H^2v||\nabla_H\Delta_Hv|+|\nabla_H^2T||\nabla_H\Delta_HT|)
  +|\nabla_Hv|(|\nabla_Hv||\nabla_H\Delta_Hv|\nonumber\\
  &+|\nabla_HT||\nabla_H\Delta_HT|)+\left(\int_{-h}^h|\nabla_Hv|dz\right)(|\nabla_Hu||\nabla_H
  \Delta_Hv|+|\nabla_H\partial_zT||\nabla_H\Delta_HT|)\nonumber\\
  &+\left(\int_{-h}^h|\nabla_H^2v|dz
  \right)(|u||\nabla_H\Delta_Hv|+(|\partial_zT|+1)|\nabla_H\Delta_HT|)\nonumber\\
  &+\left(\int_{-h}^h|\nabla_H^2T|dz\right)|\nabla_H\Delta_Hv|\bigg]dxdydz\nonumber\\
  \leq&C[\|v\|_\infty(\|\nabla_H^2v\|_2\|\nabla_H\Delta_Hv\|_2+\|\nabla_H^2T\|_2\|\nabla_H\Delta_H
  T\|_2)\nonumber\\
  &+\|\nabla_H^2v\|_2\|\nabla_H\Delta_HT\|_2+\|\nabla_H^2T\|_2\|\nabla_H\Delta_Hv\|_2]\nonumber\\
  &+C\int_M\int_{-h}^h(|\nabla_Hv|+|\nabla_Hu|)
  dz\int_{-h}^h(|\nabla_Hv||\nabla_H\Delta_Hv|+|\nabla_HT||\nabla_H\Delta_HT|)dzdxdy\nonumber\\
  &+C\int_M\int_{-h}^h|\nabla_Hv|dz
  \int_{-h}^h(|\nabla_Hu||\nabla_H\Delta_Hv|
  +|\nabla_H\partial_zT||\nabla_H\Delta_HT|)dzdxdy\nonumber\\
  &+C\int_M\int_{-h}^h|\nabla_H^2v|dz
  \int_{-h}^h(|u||\nabla_H\Delta_Hv|+|\partial_zT||\nabla_H\Delta_HT|)dzdxdy.\label{esthhvt1}
\end{align}
By Lemma \ref{lad}, one has
\begin{align}
  &C\int_M\int_{-h}^h(|\nabla_Hv|+|\nabla_Hu|)
  dz\int_{-h}^h(|\nabla_Hv||\nabla_H\Delta_Hv|+|\nabla_HT||\nabla_H\Delta_HT|)dzdxdy
  \nonumber\\
  &+C\int_M\int_{-h}^h|\nabla_Hv|dz
  \int_{-h}^h(|\nabla_Hu||\nabla_H\Delta_Hv|
  +|\nabla_H\partial_zT||\nabla_H\Delta_HT|)dzdxdy\nonumber\\
  \leq&C(\|\nabla_Hv\|_2^{\frac{1}{2}}+\|\nabla_Hu\|_2^{\frac{1}{2}})(\|\nabla_H^2v\|_2^{\frac{1}{2}}
  +\|\nabla_H^2u\|_2^{\frac{1}{2}})(\|\nabla_Hv\|_2^{\frac{1}{2}}+\|\nabla_HT\|_2^{\frac{1}{2}})
  (\|\nabla_H^2v\|_2^{\frac{1}{2}}\nonumber\\
  &+\|\nabla_H^2T\|_2^{\frac{1}{2}})(\|\nabla_H\Delta_Hv\|_2+\|\nabla_H\Delta_HT\|_2)
  +C\|\nabla_Hv\|_2^{\frac{1}{2}}\|\nabla_H^2v\|_2^{\frac{1}{2}}(\|\nabla_Hu\|_2^{\frac{1}{2}}
  \nonumber\\
  &+\|\nabla_H\partial_zT\|_2^{\frac{1}{2}})(\|\nabla_H^2u\|_2^{\frac{1}{2}}
  +\|\nabla_H^2\partial_zT\|_2^{\frac{1}{2}})(\|\nabla_H\Delta_Hv\|_2+\|\nabla_H\Delta_HT\|_2)
  \nonumber\\
  \leq&\frac{1}{8}(\|\nabla_H\Delta_Hv\|_2^2+\|\nabla_H\Delta_HT\|_2^2)+C(\|\nabla_Hv\|_2^2+
  \|\nabla_HT\|_2^2+\|\nabla_Hu\|_2^2)(\|\Delta_Hv\|_2^2\nonumber\\
  &+\|\Delta_HT\|_2^2+\|\Delta_Hu\|_2^2)+C[\|\nabla_Hv\|_2^2\|\Delta_Hv\|_2^2+(\|\nabla_Hu\|_2^2+\|
  \nabla_H\partial_zT\|_2^2)\nonumber\\
  &\times(\|\Delta_Hu\|_2^2+\|\Delta_H\partial_zT\|_2^2)]\nonumber\\
  \leq&\frac{1}{8}(\|\nabla_H\Delta_Hv\|_2^2+\|\nabla_H\Delta_HT\|_2^2)+C(\|\nabla_Hv\|_2^2+
  \|\nabla_HT\|_2^2+\|\nabla_Hu\|_2^2+\|\nabla_H\partial_zT\|_2^2)\nonumber\\
  &\times(\|\Delta_Hv\|_2^2+\|\Delta_HT\|_2^2+\|\Delta_Hu\|_2^2+\|\Delta_H\partial_zT\|_2^2).
  \label{esthhvt2}
\end{align}
Similar argument as that for (\ref{4.2-1}) yields
\begin{align}
  &\int_M \int_{-h}^h|\nabla_H^2v|dz  \int_{-h}^h|u||\nabla_H\Delta_Hv|dz dxdy
  \leq\frac{1}{8}\|\nabla_H\Delta_Hv\|_2^2
  +C\|\nabla_Hv\|_2^2\|u\|_4^8.\label{esthhvt3}
\end{align}
It follows from integration by parts and Proposition \ref{prop3.1} (ii) that
\begin{align*}
  &\int_\Omega|\partial_zT|^4dxdydz=-\int_\Omega T\partial_z(|\partial_zT|^2\partial_zT)dxdydz\\
  =&-3\int_\Omega T|\partial_zT|^2\partial_z^2Tdxdydz\leq3\|T\|_\infty\|\partial_z^2T\|_2\|\partial_zT\|_4^2
  \leq C\|\partial_z^2T\|_2\|\partial_zT\|_4^2,
\end{align*}
and thus
$$
\|\partial_zT\|_4^2\leq C\|\partial_z^2T\|_2.
$$
By the aid of this inequality, the same argument as that for (\ref{4.2-1}) yields
\begin{align}
  &\int_M\int_{-h}^h|\nabla_H^2v|dz\int_{-h}^h|\partial_zT||\nabla_H\Delta_HT|dzdxdy\nonumber\\
  \leq&\frac{1}{8}\|\nabla_H\Delta_HT\|_2^2+C\|\nabla_Hv\|_2^2\|\partial_zT\|_4^8
  \leq\frac{1}{8}\|\nabla_H\Delta_HT\|_2^2+C\|\nabla_Hv\|_2^2\|\partial_z^2T\|_2^4.\label{esthhvt4}
\end{align}
Substituting (\ref{esthhvt2})--(\ref{esthhvt4}) into (\ref{esthhvt1}), and using the Young inequality, one obtains
\begin{align*}
  &\frac{d}{dt}\|(\Delta_Hv,\Delta_HT)\|_2^2+\|(\nabla_H\Delta_Hv,
  \nabla_H\Delta_HT,\sqrt\varepsilon\partial_z\Delta_Hv,\sqrt\varepsilon
  \partial_z\Delta_HT)\|_2^2\\
  \leq&C(\|v\|_\infty^2+1)(\|\Delta_Hv\|_2^2+\|\Delta_HT\|_2^2)+C[\|\nabla_Hv\|_2^2
  (\|u\|_4^8+\|\partial_z^2T\|_2^4)]\\
  &+C(\|\nabla_Hv\|_2^2+\|\nabla_HT\|_2^2+\|\nabla_Hu\|_2^2+\|\nabla_H\partial_zT\|_2^2)
  (\|\Delta_Hv\|_2^2\\
  &+\|\Delta_HT\|_2^2+\|\Delta_Hu\|_2^2+\|\Delta_H\partial_zT\|_2^2)\\
  \leq&C(\|v\|_\infty^2+1)(\|\Delta_Hv\|_2^2+\|\Delta_HT\|_2^2)
  +C(1+\|\nabla_Hv\|_2^2+\|\nabla_HT\|_2^2\\
  &+\|\nabla_Hu\|_2^2+\|\nabla_H\partial_zT\|_2^2+\|u\|_4^4+\|\partial_z^2T\|_2^2)^3
  (1+\|\Delta_Hv\|_2^2\\
  &+\|\Delta_HT\|_2^2+\|\Delta_Hu\|_2^2+\|\Delta_H\partial_zT\|_2^2),
\end{align*}
proving (v).
\end{proof}

\section{A priori $H^2$ estimates and global well-posedness}
\label{sec5}

In this section, based on the a priori low order energy estimates established in section \ref{sec3} and the high order energy inequalities established in section \ref{sec4}, we can apply the logarithmic Sobolev embedding inequality (Lemma \ref{log}) and the system version of the Gronwall inequality (Lemma \ref{gronwall}) to obtain the a priori $H^2$ estimate for strong solutions to system (\ref{eq1})--(\ref{eq3}), subject to the boundary and initial conditions (\ref{BC1})--(\ref{IC}), and further establish the global well-posedness of strong solutions to system (\ref{1.8})--(\ref{IC}), or equivalently to system (\ref{1.1})--(\ref{1.7}).

We first focus on the a priori $H^2$ bounds for the strong solutions to the regularized system (\ref{eq1})--(\ref{eq3}) subject to the boundary and initial conditions (\ref{BC1})--(\ref{IC}).

\begin{proposition}
  \label{prop5.1}
Given a positive time $\mathcal T$. Let $(v,T)$ be the strong solution of system (\ref{eq1})--(\ref{eq3}) on $\Omega\times(0,\mathcal T)$, subject to the boundary and initial conditions (\ref{BC1})--(\ref{IC}). Then we have
\begin{align*}
  &\sup_{0\leq t\leq\mathcal T}\|(v,T)\|_{H^2(\Omega)}^2+\int_0^\mathcal{T}
  \Big(\|(\partial_tv,
  \partial_tT)\|_{H^1(\Omega)}^2\\
  &+\|(\nabla_Hv,\nabla_HT,\sqrt\varepsilon \partial_zv,\sqrt\varepsilon\partial_zT)\|_{H^2(\Omega)}^2\Big)
  dt\leq
  C(h,\mathcal T,v_0,T_0),
\end{align*}
where $C(h,\mathcal T,v_0,T_0)$ denotes a positive constant depending only on $h, \mathcal T$ and the initial data.
\end{proposition}

\begin{proof}
Define nonnegative functions $a_i$ and $b_i$, $i=1,2,\cdots,9$, as follows
\begin{align*}
  &a_1 =\|u\|_2^2 +\|u\|_4^4 ,&&
  b_1 =\|(\nabla_Hu,\sqrt\varepsilon\partial_zu,|u|\nabla_Hu)\|_2^2 ,\\
  &a_2 =\|\partial_zT\|_2^2 ,&&b_2 =\|(\nabla_H\partial_zT,
  \sqrt\varepsilon\partial_z^2T)\|_2^2 ,\\
  &a_3 =\|\nabla_Hv\|_2^2 ,&&b_3 =\|(\Delta_Hv,\sqrt\varepsilon\nabla_H\partial_zv)\|_2^2 ,\\
  &a_4 =\|\nabla_HT\|_2^2 ,&&b_4 =\|(\Delta_HT,\sqrt\varepsilon\nabla_H\partial_zT)\|_2^2 ,\\
  &a_5 =\|\partial_zu\|_2^2 ,&&b_5 =\|(\nabla_H\partial_zu,\sqrt\varepsilon\partial_z^2u)
  \|_2^2 ,\\
  &a_6 =\|\partial_z^2T\|_2^2 ,&&b_6 =\|(\nabla_H\partial_z^2T,\sqrt\varepsilon
  \partial_z^3T)\|_2^2 ,\\
  &a_7 =\|\nabla_Hu\|_2^2 ,&&b_7 =\|(\Delta_Hu,\sqrt\varepsilon\nabla_H\partial_zu)
  \|_2^2 ,\\
  &a_8 =\|\nabla_H\partial_zT\|_2^2 ,&&b_8 =\|(\Delta_H\partial_zT,\sqrt\varepsilon
  \nabla_H\partial_z^2T)\|_2^2 ,\\
  &a_9 =\|(\Delta_Hv,\Delta_HT)\|_2^2,
  &&b_9 =\|(\nabla_H\Delta_Hv,\nabla_H
  \Delta_HT,\sqrt\varepsilon\partial_z\Delta_Hv,\sqrt\varepsilon\partial_z\Delta_HT)\|_2^2.
\end{align*}
By Proposition \ref{prop4.1} and Proposition \ref{prop4.2}, we have
\begin{align*}
\frac{d}{dt}a_1+b_1\leq& C(1+\|v\|_\infty^2)a_1,\\
\frac{d}{dt}a_2+b_2\leq& C(1+\|v\|_\infty^2)a_2+C\|\nabla_Hv\|_2^2+C(a_1+b_1),\\
\frac{d}{dt}a_3+b_3\leq& C\|v\|_\infty^2a_3+C\|\nabla_HT\|_2^2+Ca_1^2,\\
\frac{d}{dt}a_4+b_4\leq& C(1+a_2+a_3)^2(1+b_2+b_3),\\
\frac{d}{dt}a_5+b_5\leq& C(1+\|v\|_\infty^2)a_5+C(b_1+b_2),\\
\frac{d}{dt}a_6+b_6\leq& C(1+\|v\|_\infty^2)a_6+C(1+a_2+a_5)^2(1+b_1+b_2+b_5),\\
\frac{d}{dt}a_7+b_7\leq& C\|v\|_\infty^2a_7+C(a_1+a_3+a_4+a_5)(b_1+b_3+b_5),\\
\frac{d}{dt}a_8+b_8\leq& C\|v\|_\infty^2a_8+C(1+a_3+a_4+a_5+a_6)^2\\
&\times(1+b_3+b_4+b_5+b_6),\\
\frac{d}{dt}a_9+b_9\leq& C(1+\|v\|_\infty^2)a_9+C(1+a_1+a_3+a_4+a_6+a_7+a_8)\\
  &\times(1+b_3+b_4+b_7+b_8).
\end{align*}
Set $A_1(t)=a_1(t)+e$ and $A_i(t)=A_{i-1}(t)+a_i(t)$, $i=2,\cdots,n$. Then one can easily check from the above inequalities that
\begin{eqnarray}
  &\frac{d}{dt}A_1(t)\leq  m(t)A_1(t),\label{gron1}\\
  &\frac{d}{dt}A_i(t)+B_i(t)\leq   m(t)A_i(t)+\zeta A_{i-1}^3B_{i-1},\quad i=2,\cdots,9,\label{gron2}
\end{eqnarray}
where
$$
m(t)=C(1+\|v\|_\infty^2(t)+\|\nabla_Hv\|_2^2(t)+\|\nabla_HT\|_2^2(t)),
$$
for a positive constant $C$.

Recalling the definitions of $a_i, A_i$, $i=1,\cdots,9$, one can easily check
$$
\sum_{i=1}^9A_i\geq\sum_{i=1}^9a_i+e\geq C\|(\nabla v,\nabla T)\|_{H^1}^2+e.
$$
And thus, by Proposition \ref{prop3.1} and the Sobolev and Poincar\'e inequalities, one obtains
\begin{align*}
  \sum_{i=1}^9A_i\geq C\|(\nabla v,\nabla T)\|_{H^1}^2+e\geq C\|(v,T)\|_{H^2}^2+e\geq C(
  \|(v,T)\|_{W^{1,6}(\Omega)}^2+e).
\end{align*}
By the aid of this, it follows from Lemma \ref{log} and Proposition \ref{prop3.1} that
\begin{align}
  m(t)=&C(1+\|v\|_\infty^2+\|\nabla_Hv\|_2^2+\|\nabla_HT\|_2^2)\nonumber\\
  \leq&C(1+\|\nabla_Hv\|_2^2+\|\nabla_HT\|_2^2)(1+\|v\|_\infty^2)\nonumber\\
  \leq&C(1+\|\nabla_Hv\|_2^2+\|\nabla_HT\|_2^2)\left[1+\left(\max\left\{1,\sup_{r\geq2}
  \frac{\|v\|_r}{\sqrt r}\right\}
  \log^{\frac{1}{2}}
  (\|v\|_{W^{1, {6}}}+e)\right)^2\right]\nonumber\\
  \leq & C(1+\|\nabla_Hv\|_2^2+\|\nabla_HT\|_2^2)
  \log(\|v\|_{W^{1, {6}}}+e)\nonumber\\
  \leq&K(t)\log \sum_{i=1}^9A_i(t) ,\label{estm}
\end{align}
where $K(t)=C(1+\|\nabla_Hv\|_2^2+\|\nabla_HT\|_2^2)$. By Proposition \ref{prop3.1}, one has $K\in L^1_{\text{loc}}([0,\infty))$.

On account of (\ref{gron1})--(\ref{estm}), one can apply Lemma \ref{gronwall} to conclude that
\begin{equation}
  \sum_{i=1}^9A_i(t)+\sum_{i=1}^9\int_0^tB_i(s)ds\leq Q(t),\label{5.4}
\end{equation}
for any $t\in(0,\mathcal T)$, where $Q$ is the corresponding continuous function on $[0,\infty)$, specified in (\ref{E2}), determined by the initial data. Recalling the definition of $A_i$ and $B_i$, one can easily check that
\begin{eqnarray*}
  &\sum_{i=1}^9A_i\geq C\|(\nabla v,\nabla T)\|_{H^1}^2,\quad\sum_{i=1}^9B_i\geq C\|(\nabla_Hv,\nabla_HT,
  \sqrt\varepsilon\partial_zv,\sqrt\varepsilon\partial_zT)\|_{H^2}^2.
\end{eqnarray*}
By the aid of the above, it follows from (\ref{5.4}) and Proposition \ref{prop3.1} that
\begin{equation*}
  \sup_{0\leq t\leq\mathcal T}\|(v,T)\|_{H^2}^2+\int_0^\mathcal{T}\|(\nabla_Hv,\nabla_HT,
  \sqrt\varepsilon\partial_zv,\sqrt\varepsilon\partial_zT)\|_{H^2}^2dt\leq C,
\end{equation*}
for a positive constant $C$ depending only on $h,\mathcal T$ and the initial data, and is independent of $\varepsilon$.

Thanks to the estimates we have just proved, one can use the same argument as in the last paragraph of the proof of Proposition 3.1 in \cite{CAOLITITI1} to obtain the corresponding estimates on $\partial_tv$ and $\partial_t T$, and thus we omit the details here. This completes the proof.
\end{proof}

After establishing the a priori $H^2$ estimate, as stated in Proposition \ref{prop5.1}, we are now ready to prove the global well-posedness of strong solutions to system (\ref{1.8})--(\ref{IC}).

\begin{proof}[\textbf{Proof of Theorem \ref{thm1}}] By Proposition \ref{lem2.4} and Proposition \ref{prop5.1}, for any $\varepsilon>0$, there is a unique global strong solution $(v_\varepsilon,T_\varepsilon)$ to system (\ref{eq1})--(\ref{eq3}) subject to the boundary and initial conditions (\ref{BC1})--(\ref{IC}), such that for any $\mathcal T>0$,
\begin{align*}
  &\sup_{0\leq t\leq\mathcal T}\|(v_\varepsilon,T_\varepsilon)\|_{H^2(\Omega)}^2+\int_0^\mathcal{T}
  \Big(\|(\partial_tv_\varepsilon,
  \partial_tT_\varepsilon)\|_{H^1(\Omega)}^2\\
  &+\|(\nabla_Hv_\varepsilon,\nabla_HT_\varepsilon,\sqrt\varepsilon \partial_zv_\varepsilon,\sqrt\varepsilon\partial_zT_\varepsilon)\|_{H^2(\Omega)}^2\Big)
  dt\leq
  C(h,\mathcal T,v_0,T_0),
\end{align*}
where $C$ is independent of $\varepsilon$.

On account of these estimates, applying Lemma \ref{AL}, there is a subsequence, still denoted by $\{(v_\varepsilon,T_\varepsilon)\}$, and $(v,T)$, such that
\begin{eqnarray*}
  &(v_\varepsilon,T_\varepsilon)\rightarrow (v,T),\quad\mbox{in }C([0,\mathcal T];H^1(\Omega)),\\
  &(\nabla_H v_\varepsilon,\nabla_HT_\varepsilon)\rightarrow(\nabla_H v,\nabla_HT),\quad\mbox{in }L^2(0,\mathcal T;H^1(\Omega)),\\
  &(v_\varepsilon,T_\varepsilon){\overset{*}{\rightharpoonup}}(v,T),\quad\mbox{in }L^\infty(0,\mathcal T;H^2(\Omega)),\\
  &(\nabla_H v_\varepsilon,\nabla_HT_\varepsilon)\rightharpoonup (\nabla_H v,\nabla_HT),\quad\mbox{in }L^2(0,\mathcal T;H^2(\Omega)),\\
  &(\partial_tv_\varepsilon,\partial_tT_\varepsilon)\rightharpoonup(\partial_tv,\partial_tT)
,\quad\mbox{in }L^2(0,\mathcal T;H^1(\Omega)),
\end{eqnarray*}
where $\rightharpoonup$ and ${\overset{*}{\rightharpoonup}}$ are the weak and weak-$*$ convergence, respectively. Thanks to these convergence, one can easily show that $(v,T)$ is a strong solution to system (\ref{1.8})--(\ref{IC}), or equivalently to system (\ref{1.1})--(\ref{1.7}).

The continuous dependence on the initial data, in particular the uniqueness, are straightforward corollary of Proposition 2.4 in \cite{CAOLITITI2}. This completes the proof.
\end{proof}
\section{Appendix: a logarithmic Sobolev embedding inequality}
\label{sec6}

In this appendix, we establish a logarithmic Sobolev embedding inequality, for any function $f\in W^{1,p}(\mathbb R^N)$, with $p>N\geq2$. Similar inequalities have been established in \cite{CAOFARHATTITI} and \cite{CAOWU} for the 2D case. We follow here the ideas of the proof presented in \cite{CAOFARHATTITI}.

\begin{lemma}
  \label{lemapp}
Let $p>N\geq2$. Then for any $F\in W^{1, {p}}(\mathbb R^N)$, we have
\begin{equation*}
  \|F\|_\infty\leq C_{N, {p},\lambda}\max\left\{1,\sup_{r\geq2}
  \frac{\|F\|_r}{r^\lambda R^{N/r}}\right\}
  \log^\lambda
  \left(e+\frac{\|F\|_p}{R^{N/p}}+\frac{\|\nabla F\|_p}{R^{N/p-1}}\right),
\end{equation*}
for any $R,\lambda>0$, and for some constant $C_{N,p,\lambda}>0$.
\end{lemma}

\begin{proof}
We only give the details of the proof for the case of spatial dimension $N\geq3$, the case that $N=2$ can be given similarly (see, e.g., \cite{CAOFARHATTITI}). Without loss of generality, we can suppose that
$|F(0)|=\|F\|_\infty$. Denote by $B_r$ the ball in $\mathbb R^N$ centered at the origin. Let $\phi\in C_0^\infty(B_1)$, with $\phi\equiv1$ on $B_{1/2}$, $0\leq\phi\leq1$ on $B_1$, and set $f=F\phi$.

Taking $\alpha$ and $\beta$ as
$$
\alpha=\frac{2Np}{p-N},\qquad\beta=\frac{2Np}{2Np-N-p},
$$
then one can easily check that
\begin{equation}\label{mukappa}
  \alpha>2N,\quad0<(N-1)\beta<N,\quad\frac{1}{\alpha}+\frac{1}{\beta}+\frac{1}{p}=1.
\end{equation}
Recall that $f$ can be represented in terms of $\Delta f$ by the Newtonian potential. By the aid of (\ref{mukappa}), for any $q\geq2$, we have
\begin{align*}
  |f(0)|^q=&C_N\left|\int_{\mathbb R^N}\frac{1}{|x|^{N-2}}\Delta\left(|f|^{q}\right)
  dx\right|\\
  =&C_N\left|\int_{B_1}\nabla\left(\frac{1}{|x|^{N-2}}\right)\cdot\nabla\left(|f|^{q}\right)
  dx\right|\\
  \leq&C_N(N-2)q\int_{B_1}\frac{|f|^{q-1}|\nabla
  f|}{|x|^{N-1}}dx\\
  \leq&C_{N}(N-2)q\|f\|_{(q-1)\alpha}^{q-1}\|\nabla f\|_{p}
  \left(\int_{B_1}\frac{dx}{|x|^{\beta(N-1)}}\right)^{\frac{1}{\beta}}\\
  \leq&C_{N,{p}}q\|f\|_{(q-1)\alpha}^{q-1}\|\nabla f\|_{p}.
\end{align*}
From the above inequality, for any $q\geq2$, noticing that $q^{\frac{1}{q}}\leq C$ and $(q-1)\alpha\geq2$, we deduce
\begin{align*}
  |f(0)|\leq&C_{N,p}\|f\|_{(q-1)\alpha}^{1-\frac{1}{q}}\|\nabla
  f\|_{p}^{\frac{1}{q}}\\
  =&C_{N, {p}} \left[\frac{\|f\|_{(q-1)\alpha}}{((q-1)\alpha)^\lambda}
  \right]^{1-\frac{1}{q}}[(q-1)\alpha]^{\lambda\left(1-\frac{1}{q}\right)}\|\nabla
  f\|_{p}^{\frac{1}{q}}\\
  \leq&C_{N, {p},\lambda} \left[\frac{\|f\|_{(q-1)\alpha}}{((q-1)\alpha)
  ^\lambda}
  \right]^{1-\frac{1}{q}}q^{\lambda}\|\nabla
  f\|_{p}^{\frac{1}{q}}\\
  \leq&C_{N, {p},\lambda}\max\left\{1,\sup_{r\geq2}\frac{\|f\|_r}{r^\lambda}\right\}q^\lambda \|\nabla
  f\|_{p}^{\frac{1}{q}},
\end{align*}
and thus
\begin{equation*}
  |f(0)|\leq C_{N, {p},\lambda}\max\left\{1,\sup_{r\geq2}\frac{\|f\|_r}{r^\lambda}\right\} \inf_{q\geq2}
  \left(q^\lambda(\|\nabla
  f\|_{p }+e^{4\lambda})^{\frac{1}{q}}\right).
\end{equation*}
One can check that
$$
\log(\|\nabla f\|_{p}+e^{4\lambda})\leq\max\{1,4\lambda\}\log(\|\nabla f\|_{p }+e)
$$
and
$$
\inf_{q\geq3}
  \left(q^\lambda(\|\nabla
  f\|_{p }+e^{4\lambda})^{\frac{1}{q}}\right)=\left(\frac{e}{\lambda}\right)^\lambda\log^\lambda
  (\|\nabla f\|_{p }+e^{4\lambda}).
$$
Therefore, we have
\begin{equation*}
  |f(0)|\leq C_{N, {p},\lambda}\max\left\{1,\sup_{r\geq2}\frac{\|f\|_r}{r^\lambda}\right\} \log^\lambda
  (\|\nabla f\|_{p }+e).
\end{equation*}
This implies
\begin{align}
\|F\|_\infty=&|f(0)|\leq C_{N, {p},\lambda}\max\left\{1,\sup_{r\geq2}\frac{\|f\|_r}{r^\lambda}\right\} \log^\lambda
  (\|\nabla f\|_{p }+e)\nonumber\\
  \leq&C_{N, {p},\lambda}\max\left\{1,\sup_{r\geq2}\frac{\|F\|_r}{r^\lambda}\right\} \log^\lambda
  (\|\nabla F\|_{p }+\|F\|_{p }+e)\nonumber\\
  \leq&C_{N, {p},\lambda}\max\left\{1,\sup_{r\geq2}\frac{\|F\|_r}{r^\lambda}\right\}
  \log^\lambda
  (\|F\|_{W^{1, {p}}}+e).\label{A1}
\end{align}

For any $R>0$, define function $F_R$ as
$$
F_R(x)=F(Rx),\quad\mbox{ for }x\in\mathbb R^N.
$$
One can easily check that
$$
\|F_R\|_p=R^{-N/p}\|F\|_p,\quad\|\nabla F_R\|_p=R^{1-N/p}\|\nabla F\|_p,
$$
for any $p\in(0,\infty)$. By the aid of the above, it follows from (\ref{A1}) that
\begin{align*}
  \|F\|_\infty=&\|F_R\|_\infty\leq C_{N,p,\lambda}\max\left\{1,\sup_{r\geq2}\frac{\|F_R\|_r}{r^\lambda}\right\}\log^\lambda(\|F_R\|_{W^{1,p}}+e)\\
  =&C_{N,p,\lambda}\max\left\{1,\sup_{r\geq2}\frac{\|F\|_r}{r^\lambda R^{N/r}}\right\}
  \log^\lambda\left(e+\frac{\|F\|_p}{R^{N/p}}+\frac{\|\nabla F\|_p}{R^{N/p-1}}\right),
\end{align*}
proving the conclusion.
\end{proof}

\section*{Acknowledgments}
{J.L.~and E.S.T.~are thankful to the warm hospitality of the  Instituto
Nacional de Mate\-m\'{a}tica  Pura  e Aplicada (IMPA), Brazil,  where part of
this work was completed. The work of C.C.~work is  supported in part by NSF grant DMS-1109022. The work of E.S.T.~is supported in part by the NSF
grants DMS-1009950, DMS-1109640 and DMS-1109645; also by the CNPq-CsF grant \#401615/2012-0, through the program Ci\^encia sem Fronteiras.}
\par

\end{document}